\documentclass[10pt]{article}
\usepackage{amsmath, amssymb, amscd, amsthm, amsfonts}
\usepackage{graphicx}
\usepackage{hyperref}
\usepackage{tikz}
\usepackage{blkarray}
\usepackage{accents}
\usepackage{algorithmic}
\usepackage{booktabs}
\usepackage[linesnumbered,ruled,vlined]{algorithm2e}

\def\sS{\mathsf{S}}

\def\sI{\mathsf{I}}
\def\sR{\mathsf{R}}
\def\sD{\mathsf{D}}

\def\lf{\mathsf{LF}}
\def\li{\mathsf{LI}}
\def\mf{\mathsf{MF}}
\def\mi{\mathsf{MI}}
\def\hf{\mathsf{HF}}
\def\hi{\mathsf{HI}}
\def\blambda{\boldsymbol{\lambda}}

\newcommand\cD{\mathcal D}

\newcommand\bp{\boldsymbol{p}}
\newcommand\bu{\boldsymbol{u}}
\newcommand{\ubar}[1]{\underaccent{\bar}{#1}}
\newcommand{\ubarc}{\underaccent{\bar}{c}}

\newtheorem{theorem}{Theorem}[section]
\newtheorem{remark}[theorem]{Remark}

\newtheorem{assumption}[theorem]{Assumption}

\newcommand\EE{\mathbb E}

\def\balpha{\boldsymbol{\alpha}}
\def\bnu{\boldsymbol{\nu}}
\def\bZ{\boldsymbol{Z}}

\DeclareMathOperator*{\argmin}{arg\,min}
\DeclareMathOperator*{\argmax}{arg\,max}
\allowdisplaybreaks

\usepackage{subcaption}
\usepackage{enumitem}
\setlist[itemize]{noitemsep, topsep=2pt, parsep=2pt, partopsep=2pt}

\title{Incorporating Authority Perception, Economic Status, and Behavioral Response in Infectious Disease Control}
\author{
Huaning Liu\thanks{Department of Statistics, University of Illinois Urbana-Champaign, Champaign, IL 61820, USA.}
\ \ \
Junke Yang\thanks{Division of Biostatistics, College of Public Health, The Ohio State University, Columbus, OH 43210, USA.}
\ \ \
Soren L. Larsen\thanks{Department of Demography, University of California Berkeley, Berkeley, CA 94720, USA.}
\ \ \
Pamela P. Martinez\thanks{Department of Microbiology, University of Illinois Urbana-Champaign, Urbana, IL 61801, USA.}\ \thanks{Carl R. Woese Institute for Genomic Biology, University of Illinois Urbana-Champaign, Urbana 61801, USA.}
\ \ \
G\"{o}k\c{c}e Dayan{\i}kl{\i}\footnotemark[1]\ \footnotemark[5] \ \thanks{Corresponding author. Email address: {\tt gokced@illinois.edu}}
}
\date{}

\begin{document}

\maketitle

\begin{abstract}
We introduce a multi-population mean field game framework to examine how economic status and authority perception shape vaccination and social distancing decisions under different epidemic control policies. We carried out a survey to inform our model and stratify the population into six groups based on income and perception of authority, capturing behavioral heterogeneity. Individuals adjust their socialization and vaccination levels to optimize objectives such as minimizing treatment costs, complying with social-distancing guidelines if they are authority-followers, or reducing losses from decreased social interactions if they are authority-indifferents, alongside economic costs. Public health authorities influence behavior through social-distancing guidelines and vaccination costs. We characterize the Nash equilibrium via a forward–backward differential equation system, provide its mathematical analysis, and develop a numerical algorithm to solve it. Our findings reveal a trade-off between social-distancing and vaccination decisions. Under stricter guidelines that target both susceptible and infected individuals, followers reduce both socialization and vaccination levels, while indifferents increase socialization due to followers’ preventative measures. Adaptive guidelines targeting infected individuals effectively reduce infections and narrow the gap between low- and high-income groups, even when susceptible individuals socialize more and vaccinate less. Lower vaccination costs incentivize vaccination among low-income groups, but their impact on disease spread is smaller than when they are coupled with social-distancing guidelines. Trust-building emerges as a critical factor in epidemic mitigation, underscoring the importance of data-informed, game-theoretical models that aim to understand complex human responses to mitigation policies.
\end{abstract}

\section*{Introduction}
Epidemic mitigation policies have been crucial strategies for controlling outbreaks, as evidenced throughout the COVID-19 pandemic \cite{flaxman2020estimating, eyre2022effect, Ruggeri2024}. During this period, individuals navigated multiple decision processes, such as changing their contact rates, mask use, and vaccine uptake, that varied systematically with demographic and economic factors \cite{Mena2021, Badillo2021, Levy2022a, Larsen2023, Manna2024}. Disentangling the role of host heterogeneity in decision-making and adherence to various public health policies is an active area of research with significant economic and public health implications. While population-level mathematical models have explored the interplay between epidemic dynamics and these variables (e.g.\cite{Bai2021-rn, Mahmud2025-dx, albi2021modelling, snellman2022modelling, jia2025learning}), these models do not account for individual-level decision-making under varying mitigation policies. Moreover, mitigation strategies are not always followed by individuals \cite{Li2021, Levy2022b}, which makes it difficult to measure their impact on disease transmission, with important implications for the effectiveness of disease control.

Classical population-level models of transmission, such as the susceptible-infected-recovered (SIR) model, do not incorporate rational decision-making at the individual-level in response to disease incidence. Rather than modeling only the population-level effects of public health policies, it is imperative to include decisions individuals make about vaccination, social distancing, and mask use. Incorporating these choices and the interactions among individuals is essential for further understanding the complexity of epidemic spread and for implementing interventions that effectively limit transmission. The emergent behavior of the dynamical system of interacting decision-makers arises from an equilibrium in which no individual has an incentive to change their behavior.

In this work, we evaluate the effect of different public health policies for rational-decision makers by implementing a mean field games (MFGs) approach. MFGs have been used in many settings, such as energy market modeling and climate change policy decisions~\cite{aid2020entry, djehiche2020price, carmona2022mean, alasseur2020extended, bagagiolo2014mean, bichuch_acc}, traffic management~\cite{chevalier2015micro, huang2021dynamic, festa2018mean}, advertisement decisions~\cite{carmona2021mean, salhab2022dynamic}, price formation~\cite{lachapelle2016efficiency, gomessaude2021meanprice}, and systemic risk in banking~\cite{carmonafouquesun2015mean}. In recent years, MFGs and their extensions have been used to model epidemic spread among rational individuals~\cite{aurell2022optimal,aurell2022finite,bremaud2024mean,olmez2022modeling,lee2020controlling,elie2020contact,cho2020mean,Doncel_Gast_Gaujal_2022,pnas_epidemics}. However, none of these studies have considered in their models (1) the joint rational decision-making of individuals concerning both vaccination and social distancing, and (2) population  stratification based on economic status and perceptions of authority.

We defined a population split into 6 groups based on 3 income levels and whether or not individuals would like to follow government recommendations. These groups were determined based on findings from a recent survey conducted in Illinois, which aimed to characterize how perceived risk influences individual behavior. This model also takes into account individuals' choices with respect to non-pharmaceutical interventions and vaccine uptake, which allows the quantification of disease outcomes under different control policies. 

Our main results reveal a trade-off between vaccine prioritization and social distancing, as individuals tend to lower their vaccination uptake when strict social-distancing guidelines are imposed. The findings further reveal free-rider behavior among the indifferent individuals under strict social distancing policies, allowing them to engage in higher socialization levels by benefiting from the overall protective behavior of recommendation-followers. These results emphasize the importance of considering decision-making of individuals instead of assuming compliance when implementing infectious disease models to inform public health agencies.

\section*{Multi-Population Mean Field Game}
To identify emerging properties across different epidemic mitigation strategies, we first need to determine how individuals' behaviors reach an equilibrium response to these policies. One of the most commonly used equilibrium notions is the Nash equilibrium, in which no individual has a preferable deviation from their current behavior. This means that if any individual changes their behavior while others remain unchanged, their personal cost increases. In this way, the Nash equilibrium represents a \textit{stable} equilibrium. However, computing a Nash equilibrium becomes increasingly challenging as the number of individuals grows, due to the exponential increase in pairwise interactions. The difficulty is further amplified in dynamic and stochastic settings, such as the models implemented in this work. 

To overcome this challenge, we use mean field games, which allow mean field approximations by assuming the population size approaches infinity~\cite{huang2007large, lasry2006jeux,lasry2007mean}.  In these games, individuals are assumed to be identical and interact through mean field interactions (i.e., population statistics such as distribution or average or proportion of states, controls, or both). As the population grows, the influence of any single individual becomes negligible, and deviations in their behavior do not affect the population statistics. Since the individuals are identical, this allows us to focus on a \textit{representative individual} and their interaction with the population through the population statistics. 

Finding equilibrium in this framework involves satisfying two conditions. First, the representative individual should choose their \textit{best response} control, such as the optimal socialization and vaccination levels that minimize their personal cost over a specific time horizon given the population statistics (e.g., the proportion of infected individuals over time). Second, since all individuals are identical, the population statistics must evolve according to the dynamics driven by the best response controls of representative individuals. This fixed point methodology yields an approximate Nash equilibrium that is characterized by a forward-backward differential equations system~\cite{Bensoussan_Book, CarmonaDelarue_book_I}. To incorporate heterogeneity among the individuals, we extend this framework to multi-population MFGs, in which the population is divided into subgroups. Individuals within the same group share identical model parameters, while individuals in different groups can differ in their characteristics~\cite{cirant2015multi, bensoussan2018meanmulti, 9148985, dayanikli2024multi}. Our approach uniquely incorporates survey responses to calibrate and parametrize these mathematical models, integrating theory with empirical behavioral choices.

\section*{Model}

\subsection*{Population Structure}

Following the survey results described in the end of the Model section and Supporting Information (SI), we consider a large (mathematically infinite) population stratified by economic status (\textit{low}, \textit{middle}, \textit{high}) and perception of authority (\textit{follower} or \textit{indifferent}), resulting in six distinct groups:
\begin{itemize}
    \item Low-Follower ($\lf$), Low-Indifferent ($\li$),
    \item Middle-Follower ($\mf$), Middle-Indifferent ($\mi$),
    \item High-Follower ($\hf$), High-Indifferent ($\hi$).
\end{itemize}
Each group is characterized by its own behavioral model and epidemiological parameters and we assume that individuals within the same group are homogeneous. Individuals control their level of social distancing (i.e.   \textit{socialization level}), which also includes other non-pharmaceutical prevention behaviors such as mask wearing, and \textit{vaccination level} to minimize their personal costs, where these two metrics vary by group. We use the following notation to denote the set of all groups: $[K]=\{\lf, \li, \mf, \mi, \hf, \hi\}$. The group proportions are denoted by $m^k$, for all $k\in[K]$ where $m^k$ is positive and $\sum_{k\in[K]} m^k=1$.

\subsection*{Infection States and Controls}
Each individual transitions among three infection states: Susceptible ($\sS$), Infected ($\sI$), and Recovered ($\sR$). Let $T>0$ be the terminal time of the modeling horizon, where $T$ can represent a season, a year, or similar period. Since the groups are homogeneous, we can focus on a \textit{representative} individual in each group. For a \textit{representative} individual in group $k$, the controls at time $t\in[0,T]$ are denoted as follows:
\begin{itemize}
    \item Socialization level: $\alpha_t^k\in[0,1]$. When an individual chooses a socialization level that is equal to 1, they keep their behavioral patterns pre-epidemic. If they choose a lower value, it means that they are taking higher non-pharmaceutical measures, such as increased social distancing indoors and/or outdoors, mask use, and shopping online instead of in person.
    \item Vaccination level: $\nu_t^k\geq0$, applied only in the susceptible state.
\end{itemize}

We define: $\balpha:= (\alpha_t^k)_{t\in[0,T], k\in[K]}$, $\balpha^k = (\alpha_t^k)_{t\in[0,T]}$. $\bnu$ and $\bnu^k$ are also defined similarly.
Socialization levels are Markovian functions, in other words, they are functions of the individuals' infection states. This means individuals choose their socialization levels depending on time and also whether they are in state $\sS, \sI,$ or $\sR$. Therefore, we can also write $\alpha_t^k= \big(\alpha_t^k (\sS), \alpha_t^k (\sI), \alpha_t^k (\sR)\big)$.

\subsection*{State Transitions}
The transitions of a representative individual in group $k$ between infection states are governed by \textit{controlled} continuous-time Markov chain with rates:
\begin{itemize}
    \item $\sS \rightarrow \sI$: Rate $\beta^k \alpha_t^k Z_t^k$, where $\beta^k$ represents the base transmission rate, $Z_t^k$ is the weighted average socialization level of infected individuals across all groups:
    \begin{equation*}
        Z_t^k = \sum_{l \in [K]} w(k, l) \bar{\alpha}_t^l(\sI)  p_t^{l}(\sI) m^l,
    \end{equation*}
    Here, $w(k,l) \in[0,1]$ gives the connection strength between groups $k$ and $l$, $\bar{\alpha}_t^l(\sI)$ gives the average socialization level of infected individuals in group $l$, $ p_t^{l}(\sI)$ is the proportion of infected individuals in group $l$, and $m^l$ is the group proportion. This formulation corresponds to the infinite population limit of the following finite individual formulation:
    \begin{equation*}
        Z_t^i = 1/N\sum_{j \in \{1,\dots, N\}} w(i,j) \alpha_t^j(\sI) \mathbf{1}_{\sI}(X_t^j)
    \end{equation*}
    as the number of individuals in the population, $N$, goes to infinity. $\mathbf{1}_{\sI}(X_t^j)$ represents the indicator function that is equal to 1 if the $j^{\rm th}$ individual is infected at time $t$. In this finite population formulation, we assume that the population of size $N$ is divided into $K$ groups and if individuals $i$ and $j$ are in groups \textsc{i} and \textsc{j}, they are connected to each other with connection strength $w(\textsc{i},\textsc{j})$. Therefore, the transition rate of an individual in group $k$ depends on (1) the base rate $\beta^k$, (2) the individual's own socialization level decision $\alpha_t^k$, and (3) the individual's interactions with other infected individuals across all the populations $Z_t^k$.
    \item $\sS \rightarrow \sR$: Rate $\kappa^k\nu_t^k$, where $\kappa^k$ represents the vaccination efficacy.
    \item $\sI \rightarrow \sR$: Rate $\gamma^k$, where $\gamma^k$ represents the recovery rate.
    \item $\sR \rightarrow \sS$: Rate $\eta^k$, where $\eta^k$ represents the waning of immunity rate.
\end{itemize}

A representative individual in group $k$ transitions from state $\sI$ to $\sR$ after an exponentially distributed time with rate $\gamma^k$. The transition rate from state $\sS$ to $\sI$ depends on the individual's own socialization level decision as well as the weighted average socialization level of infected individuals across all groups. This weighted average introduces interactions among the individuals in the model. A susceptible individual can choose a lower socialization level to slow down their transition to the infected state. However, this transition rate still depends on how other infected individuals behave. The transition rate from state $\sS$ to $\sR$ is based on the individual's own vaccination level decision. A more detailed discussion about the connection with the SIR model and mean field transitions can be found in the SI. 

\begin{figure}[h]
\begin{center}

\vspace{-0.2cm} 
\tikzset{every picture/.style={line width=0.5pt}}

\begin{tikzpicture}[x=0.65pt,y=0.6pt,yscale=-0.9,xscale=0.9]

\draw  [color={rgb, 255:red, 0; green, 0; blue, 0 }  ,draw opacity=1 ][fill={rgb, 255:red, 124; green, 180; blue, 251 }  ,fill opacity=1 ] (63,110.77) .. controls (63,105.37) and (67.37,101) .. (72.77,101) -- (108.23,101) .. controls (113.63,101) and (118,105.37) .. (118,110.77) -- (118,140.07) .. controls (118,145.47) and (113.63,149.84) .. (108.23,149.84) -- (72.77,149.84) .. controls (67.37,149.84) and (63,145.47) .. (63,140.07) -- cycle ;

\draw  [fill={rgb, 255:red, 184; green, 233; blue, 134 }  ,fill opacity=1 ] (365,109.77) .. controls (365,104.37) and (369.37,100) .. (374.77,100) -- (410.23,100) .. controls (415.63,100) and (420,104.37) .. (420,109.77) -- (420,139.07) .. controls (420,144.47) and (415.63,148.84) .. (410.23,148.84) -- (374.77,148.84) .. controls (369.37,148.84) and (365,144.47) .. (365,139.07) -- cycle ;

\draw  [fill={rgb, 255:red, 244; green, 190; blue, 102 }  ,fill opacity=1 ] (221,109.77) .. controls (221,104.37) and (225.37,100) .. (230.77,100) -- (266.23,100) .. controls (271.63,100) and (276,104.37) .. (276,109.77) -- (276,139.07) .. controls (276,144.47) and (271.63,148.84) .. (266.23,148.84) -- (230.77,148.84) .. controls (225.37,148.84) and (221,144.47) .. (221,139.07) -- cycle ;

\draw    (118,128.84) -- (220,128.84) ;
\draw [shift={(222,128.84)}, rotate = 180] [color={rgb, 255:red, 0; green, 0; blue, 0 }  ][line width=0.75]    (10.93,-3.29) .. controls (6.95,-1.4) and (3.31,-0.3) .. (0,0) .. controls (3.31,0.3) and (6.95,1.4) .. (10.93,3.29)   ;

\draw    (276,129.84) -- (364,129.84) ;
\draw [shift={(366,129.84)}, rotate = 180] [color={rgb, 255:red, 0; green, 0; blue, 0 }  ][line width=0.75]    (10.93,-3.29) .. controls (6.95,-1.4) and (3.31,-0.3) .. (0,0) .. controls (3.31,0.3) and (6.95,1.4) .. (10.93,3.29)   ;

\draw    (88,100.84) .. controls (162.63,53.08) and (321.4,56.81) .. (386.79,96.4) ;
\draw [shift={(387.77,97)}, rotate = 211.8] [color={rgb, 255:red, 0; green, 0; blue, 0 }  ][line width=0.75]    (10.93,-3.29) .. controls (6.95,-1.4) and (3.31,-0.3) .. (0,0) .. controls (3.31,0.3) and (6.95,1.4) .. (10.93,3.29)   ;

\draw    (88,153.84) .. controls (162.63,203.08) and (321.4,206.81) .. (386.79,150.4) ;
\draw [shift={(82.77,151)}, rotate = 25.8] [color={rgb, 255:red, 0; green, 0; blue, 0 }  ][line width=0.75]    (10.93,-3.29) .. controls (6.95,-1.4) and (3.31,-0.3) .. (0,0) .. controls (3.31,0.3) and (6.95,1.4) .. (10.93,3.29)   ;

\draw (82,115) node [anchor=north west][inner sep=0.75pt]  [font=\Large] [align=left] {$\sS$};
\draw (244,115) node [anchor=north west][inner sep=0.75pt]  [font=\Large] [align=left] {$\sI$};
\draw (384,115) node [anchor=north west][inner sep=0.75pt]  [font=\Large] [align=left] {$\sR$};
\draw (125,102) node [anchor=north west][inner sep=0.75pt]   [align=left] {$\displaystyle \beta^k\alpha_t^k(\sS) Z_t^k$};;
\draw (305,102) node [anchor=north west][inner sep=0.75pt]   [align=left] {$\displaystyle \gamma^k $};
\draw (218,40) node [anchor=north west][inner sep=0.75pt]   [align=left]  {$\displaystyle \kappa^k \nu^k_t $};
\draw (235,170) node [anchor=north west][inner sep=0.75pt]   [align=left]  {$\displaystyle \eta^k$};

\end{tikzpicture}
\end{center}
\vspace{-0.4cm} 
\caption{\small Transition rates between states of the representative individual $k$.} 
\vspace{-0.3cm} 
\label{fig:transitionrates_mfg}
\end{figure}
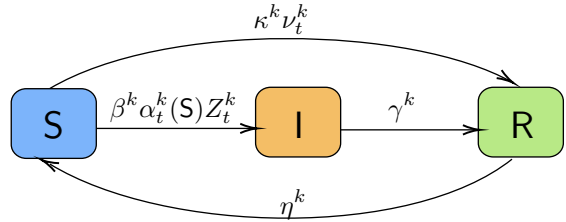

\subsection*{Public Health Policies}

Public health recommendations influence individual behavior through two primary mechanisms:
\vskip1mm
\noindent \textit{Social Distancing Guidelines:} These are encoded as time-dependent functions depending on the infection state and are denoted by $\lambda_t^{k,e}\in[0,1]$ for $t\in[0,T]$, $k\in[K]$, $e\in\{\sS, \sI, \sR\}$. A higher value of $\lambda$ indicates no restrictions, while lower values represent stricter non-pharmaceutical policies (e.g., lockdowns, mask mandates, quarantine). We stress that individuals do not have to follow these policies \textit{exactly} even if they are followers. In other words, the socialization level that is chosen by them does not have to be equal to the social distancing guidelines. Instead, it will be chosen by the individuals such that they settle in a Nash equilibrium. This means that each of them will choose a socialization level such that if they change their behavior (while everyone else keeps their behavior the same), they will have a higher personal cost. 

\vskip1mm
\noindent \textit{Vaccination Cost Modulation:} The regulator can indirectly influence vaccination uptake by adjusting the economic burden associated with vaccination and by affecting the perceived risks from vaccination uptake, represented by the group-specific coefficient $c_\nu^k>0$. For example, reducing logistical barriers or offering financial incentives can lower $c_\nu^k$, especially for low-income groups, thereby encouraging higher vaccination rates. These two policies allow public health authorities to shape disease dynamics by targeting behavioral incentives across the heterogeneous population.

\subsection*{Individual Objectives}
Each individual minimizes their own cost functional over time $[0,T]$, which depends on their group.
\vskip1mm
\noindent \textit{Authority followers}, $k \in \{\lf, \mf, \hf \}$:
These individuals try to align their socialization behavior with public health policies. Their cost increases when they deviate from social distancing guidelines, $\lambda_t^{k,e}$. This captures both compliance costs (e.g., fines for over-socializing) and opportunity costs (e.g., missed activities due to under-socializing). Mathematically, the representative authority follower individual would like to minimize the following cost functional $J^k(\balpha^k, \bnu^k; \bZ)$ by choosing their socialization and vaccination levels:
\begin{equation}
\begin{aligned}
\label{eq:obj_follow}
        J^k(\balpha^k, \bnu^k; \bZ)&=  \mathbb{E}\Big[\int_0^T \big[\big(c^k_{\lambda} (\lambda_t^{k,\sS} -\alpha_t^k)^2 + c_{\nu}^k (\nu_t^k)^2 \big)\mathbf{1}_{\sS}(X^k_t) \\
    &\qquad+ \big((\lambda_t^{k,\sI} -\alpha_t^k)^2 + c_{\sI}^k\big)\mathbf{1}_{\sI}(X^k_t)+ (\lambda_t^{k,\sR} -\alpha_t^k)^2\mathbf{1}_{\sR}(X^k_t)\big] dt \Big],
\end{aligned}
\end{equation}
where $\mathbf{1}_e(X^k_t)$ represents the indicator function for $e\in\{\sS, \sI, \sR\}$ that is equal to 1 if $X^k_t=e$ and 0 otherwise. The expression $(\lambda_t^{k, e}-\alpha_t^k)^2$ aims to capture the cost that follower individuals experience when they choose socialization levels $\alpha_t^k$ that deviate from the social distancing guidelines $\lambda_t^{k,e}$.

\vskip1mm
\noindent \textit{Authority indifferents}, $k \in \{\li, \mi, \hi\}$: These individuals do not adhere to the official guidelines and instead aim to maximize social interactions. This means that if they are susceptible or recovered, they want to choose socialization levels close to $1$ (i.e., the maximum socialization level) and if they are infected they would like to choose a socialization level that is slightly lower than $1$ (which will be denoted by $\xi_{\sI}$ and termed as \textit{sickness related intrinsic socialization}) to take into account any personal desire to decrease interactions during infection. Therefore, their cost function penalizes any reduction from maximum socialization (or any deviation from the sickness related intrinsic socialization), reflecting personal or economic losses from limited contact. 

Mathematically, the representative authority indifferent individual would like to minimize the following cost functional $J^k(\balpha^k, \bnu^k; \bZ)$ by choosing their socialization and vaccination levels:
\begin{equation}
\begin{aligned}
\label{eq:obj_indiff}
    J^k(\balpha^k, \bnu^k; \bZ) &=  \mathbb{E}\Big[\int_0^T \big[\big(c^k_{\lambda} (1 -\alpha_t^k)^2 + c_{\nu}^k (\nu_t^k)^2 \big)\mathbf{1}_{\sS}(X^k_t) \\
    &\qquad+ \big((\xi_{\sI} -\alpha_t^k)^2 + c_{\sI}^k\big)\mathbf{1}_{\sI}(X^k_t)+ (1 -\alpha_t^k)^2\mathbf{1}_{\sR}(X^k_t)\big] dt \Big]. 
\end{aligned}
\end{equation}
Susceptible and recovered individuals incur a cost if their socialization level, $\alpha_t^k$, deviates from $1$, due to the cost terms $(1-\alpha_t^k)^2$; while the infected individuals incur a cost if their socialization level, $\alpha_t^k$, deviates from the sickness related intrinsic socialization, $\xi_{\sI}$ given by $(\xi_{\sI}-\alpha_t^k)^2$.

As shown in~\eqref{eq:obj_follow} and~\eqref{eq:obj_indiff}, authority followers and indifferents also incur:
\begin{itemize}
    \item Vaccination costs for susceptible individuals, weighted by $c_\nu^k$ , that may vary by economic status and authority perception (e.g., higher for low-income groups due to access barriers) and will depend on their vaccination level controls, $\nu_t^k$. 
    \item Treatment costs for infected individuals, represented by $c_{\sI}^k$, which also vary by economic status (e.g., due to lack of insurance). 
\end{itemize}

In addition, model parameter $c_\lambda^k$ weighs the importance a susceptible individual gives to following social distancing guidelines versus socialization. These cost functional structures allow us to capture how trust in authority and economic constraints jointly shape epidemic-related behaviors. We stress that in our mathematical model we use group-dependent parameters for maximum generality of the model.

\subsection*{Equilibrium Notion}
$(\hat{\balpha}, \hat{\bnu})$ is a multi-population mean field game Nash equilibrium socialization and vaccination level control couple if for any $\balpha^k$, $\bnu^k$, we have
\begin{equation*}
    J^k(\balpha^k, \bnu^k; \hat{\bZ})\geq J^k(\hat{\balpha}^k, \hat{\bnu}^k; \hat{\bZ}),  
\end{equation*}
for all $k\in[K]$ and where  $\hat{\bZ} = (\hat{Z}_t^k)_{t\in[0,T], k\in[K]}$ and $\hat{Z}_t^k = \sum_{l \in [K]} w(k, l) \bar{\hat{\alpha}}_t^l(\sI) p_t^{l}(\sI)m^l$.
\vskip2mm
Intuitively, this means that none of the individuals can decrease their individual costs by changing their socialization and vaccination levels while the other individuals keep their behavior the same. In this way, Nash equilibrium defines a stable equilibrium where no individual has any incentive to change their behavior. The multi-population mean field game Nash equilibrium satisfies the following two conditions:
\begin{itemize}
    \item \textit{Optimality:} For each representative individual in group $k\in[K]$, $\hat{\balpha}^k$ should be the \textit{best response} control given the weighted average socialization levels of infected individuals across all groups, $\hat{\bZ}$:
    \begin{equation*}
        (\hat{\balpha}^k, \hat{\bnu}^k) = \argmin_{\balpha^k, \bnu^k} J^k({\balpha}^k, {\bnu}^k; \hat{\bZ}).
    \end{equation*}
    This implies that each representative individual chooses socialization and vaccination levels that minimize their own cost function given the other individuals' behaviors.
    \item \textit{Fixed Point:} Weighted average socialization level of infected individuals across all groups, $\hat{\bZ}^k$ for all $k\in[K]$, should be determined by the best response controls of the representative individuals:
    \begin{equation*}
        \hat{Z}_t^k = \sum_{l \in [K]} w(k, l) \bar{\hat{\alpha}}_t^l(\sI) p_t^{l}(\sI)m^l.
    \end{equation*}
    This means that at the equilibrium \textit{every individual} chooses socialization and vaccination levels that minimize their own cost functions. Assuming every individual in the same group is homogeneous, we calculate the weighted average socialization level of infected individuals across all groups by using the best response socialization levels.
\end{itemize}

\subsection*{Survey}
We conducted a two-wave online survey in Illinois between November 2024 and March 2025, recruiting respondents through the Dynata panel~\cite{Dynata}  to report on their respiratory illness-related choices over the past year. Informed consent was obtained for all participants, and the study protocol received an IRB exemption determination (reference number IRB24-1168). This survey allowed us identify factors influencing health-seeking behavior, willingness to get tested or vaccinated, and health history across various sociodemographic groups. We incorporated survey results into our model parametrization as described in the SI.

\section*{Results and Discussion}

\subsection*{Equilibrium Characterization}

Under continuity assumptions on $\lambda^{k,e}_t$ over time, the equilibrium socialization and vaccination levels in the groups are characterized as follows:

\paragraph{Equilibrium Socialization Levels:}
\begin{equation*}
\begin{aligned}
\hat{\alpha}^k_t(\sS) &= 
\begin{cases}
\lambda^{k,\sS}_t + \frac{\beta_k Z_t^k (u_t^k(\sS) - u_t^k(\sI))}{2c^k_\lambda}, & \text{if } k \in \{\lf, \mf, \hf\}, \\[2mm]
1 + \frac{\beta_k Z_t^k (u_t^k(\sS) - u_t^k(\sI))}{2c^k_\lambda}, & \text{if } k \in \{\li, \mi, \hi\},
\end{cases}\\
\hat{\alpha}^k_t(\sI) &= 
\begin{cases}
\lambda^{k,\sI}_t , & \text{if } k \in \{\lf, \mf, \hf\}, \\%[2mm]
\xi_{\sI}, & \text{if } k \in \{\li, \mi, \hi\},
\end{cases}\\
\hat{\alpha}^k_t(\sR) &= 
\begin{cases}
\lambda^{k,\sR}_t , & \text{if } k \in \{\lf, \mf, \hf\}, \\%[2mm]
1, & \text{if } k \in \{\li, \mi, \hi\},
\end{cases}
\end{aligned}
\end{equation*}
where $Z_t^k=\sum_{l \in [K]} w(k, l) \bar{\hat{\alpha}}_t^l(\sI) p_t^{l}(\sI)m^l$,
\paragraph{Equilibrium Vaccination Levels:}
\begin{equation*}
\hat{\nu}^k_t = \frac{\kappa^k \big(u_t^k(\sS) - u_t^k(\sR)\big)}{2c_\nu^k},
\end{equation*}
where the expected minimum cost of the representative individual in group $k$ who is in state $e$ at time $t$ ($(u_t^k(e))_{t\in[0,T], k\in [K], e\in\{\sS, \sI, \sR\}}$), and the proportion of individuals who are in state $e$ in group $k$ at time $t$ ($(p_t^k(e))_{t\in[0,T], k\in [K], e\in\{\sS, \sI, \sR\}}$) satisfy the following coupled forward-backward differential equation (FBODE) system:
\begin{align*}
    \dot p^k_t(\sS) &= -\beta^k \hat{\alpha}^{k}_t(\sS) Z_t^k p_t^k(\sS) - \kappa^k \hat{\nu}_t^kp_t^k(\sS) + \eta^k p_t^k(\sR) ,\\[1mm]
    \dot p^k_t(\sI) &= \beta^k \hat{\alpha}^{k}_t(\sS) Z_t^k p_t^k(\sS) - \gamma^k p_t^k(\sI),\\[1mm]
    \dot p^k_t(\sR) &= \gamma^k p_t^k(\sI) +\kappa^k \hat{\nu}_t^kp_t^k(\sS) - \eta^k p_t^k(\sR),
    \\[1mm]
    \dot u^{k_1}_t (\sS)&= \beta^{k_1} \hat\alpha^{k_1}_t(\sS)Z^{k_1}_t\big(u^{k_1}_t(\sS) - u^{k_1}_t( \sI)\big) \\
    &\hskip5mm+ \kappa^{k_1} \hat{\nu}_t^{k_1} \big(u^{k_1}_t(\sS) - u^{k_1}_t( \sR)\big) 
    -c^{k_1}_\lambda \big(\lambda^{{k_1}, \sS}_t-\hat\alpha^{k_1}_t(\sS)\big)^2 - c_\nu^{k_1} (\hat{\nu}_t^{k_1})^2,
    \\[1mm]
    \dot u^{k_2}_t (\sS)&= \beta^{k_2} \hat\alpha^{k_2}_t(\sS)Z^{k_2}_t\big(u^{k_2}_t(\sS) - u^{k_2}_t( \sI)\big) \\
    &\hskip5mm+ \kappa^{k_2} \hat{\nu}_t^{k_2} \big(u^{k_2}_t(\sS) - u^{k_2}_t( \sR)\big) 
    -c^{k_2}_\lambda \big(1-\hat\alpha^{k_2}_t(\sS)\big)^2 - c_\nu^{k_2} (\hat{\nu}_t^{k_2})^2,
    \\[1mm]
    \dot u^k_t (\sI)&= \gamma^k \big(u^k_t(\sI) - u^k_t( \sR)\big) - c^k_{\sI},
    \\[1mm]
    \dot u^k_t (\sR)&= \eta^k (u_t^k(\sR)-u_t^k(\sS)),
    \\[1mm]
     u^k_T(e)&= 0,\quad p^k_0(e) = p^k(e),\quad \forall e\in \{\sS, \sI, \sR\},  \\[1mm]
     Z^k_t &= \sum_{l \in [K]} w(k, l) \EE[\hat{\alpha}_t^l(\sI)] p_t^l(\sI) m^l, \quad\forall k \in [K],\, \forall t\in[0,T], \forall k_1\in\{\lf,\mf, \hf\},\: \forall k_2\in\{\li, \mi, \hi\}.
\end{align*}
The above result indicates that if the FBODE system is solved, then its solution $\big(u_t^k(e),p_t^k(e)\big)_{t\in[0,T], k\in[K], e\in\{\sS, \sI, \sR\}}$ can be used to calculate the equilibrium socialization and vaccination levels. In order to numerically solve this FBODE system, a fixed point point algorithm is implemented. The full equilibrium characterization theorem, its intuitive details and proof, and the existence and uniqueness of the multi-population mean field Nash equilibrium theorem as well as its proof are presented in the SI.

\subsection*{Numerical Algorithm}
To numerically compute the multi-population mean field Nash equilibrium, we solve the FBODE system above characterizing the equilibrium. We discretize the time using Euler scheme to write discrete-time dynamics and obtain the equilibrium by implementing a fixed-point algorithm to iteratively update the forward and backward components of the discretized ODE system until convergence. More details and the detailed pseudo-code of the algorithm is provided in the SI.

\subsection*{Experiment Results and Observations}
We define three metrics to analyze our results. \textit{Peak difference}, which quantifies the difference between the peak of the proportion of infected individuals or the percentage difference between the peak in vaccination or socialization levels in the same group \textit{under two different experiment settings}. Mathematically, the infection peak difference in group $k$ is defined as 
$|\max_t p_t^{k, \text{exp 1}}(\sI) - \max_t p_t^{k, \text{exp 2}} (\sI)|$. On the other hand, the (susceptible) socialization and vaccination peak difference in group $k$ are defined respectively as 
$$|\min_t \hat{\alpha}_t^{k, \text{exp 1}}(\sS) - \min_t \hat{\alpha}_t^{k, \text{exp 2}}(\sS)|$$ 
and 
$$|\max_t \hat\nu_t^{k, \text{exp 1}} - \max_t \hat\nu_t^{k, \text{exp 2}}|.$$
\textit{Peak time span} refers to the time step difference between the earliest and latest peak across all groups under the same policy, that is mathematically defined as
$$\max_{k}\argmax_{t} p_t^k(\sI) - \min_{k}\argmax_{t} p_t^k(\sI)$$
\textit{Group disparity} refers to the maximum difference between the proportions of infected individuals or between socialization levels or between vaccination levels in two different groups over time under the same policies. Mathematically, the infection group disparity between group $k$ and group $l$ is given by $\max_t \big(p_t^{k}(\sI)-p_t^l(\sI)\big)$. Similarly, the (susceptible) socialization and vaccination group disparity between group $k$ and group $l$ under the same experiment setting are defined as $\max_t \big(\alpha_t^{k}(\sS)-\alpha_t^l(\sS)\big)$ and $\max_t \big(\nu_t^{k}-\nu_t^l\big)$, respectively.

\subsubsection*{Effects of Social Distancing Guidelines}
We compare two social distancing policies to a baseline guideline:
\begin{itemize}
    \item \textbf{Permissive (baseline):} This is the baseline policy where the public health authority does not impose any strict social distancing or other non-pharmaceutical interventions. They just recommend more caution than usual. Mathematically, we set $\lambda^{k,e}_t=0.9$ for all $t\in[0,T]$, $k\in\{\lf, \mf, \hf\}$ and $e\in\{\sS, \sI, \sR\}$.
    \item \textbf{Adaptive:} In the adaptive social distancing guidelines, the public health authority recommends a stricter intervention for individuals who are \textit{infected}. This can be implemented by suggesting preventive measures such as quarantine or mask-use for the infected individuals. Mathematically, we keep $\lambda^{k,e}_t=0.9$ for $e\in\{\sS, \sR\}$ and set $\lambda^{k,\sI}_t=0.6$ for all $t\in[0,T]$, $k\in\{\lf, \mf, \hf\}$.
    \item \textbf{Strict:} In the strict social distancing guidelines, the public health authority recommends stricter intervention for \textit{all} individuals. Mathematically, we set $\lambda^{k,e}_t=0.6$ for all $t\in[0,T]$, $k\in\{\lf, \mf, \hf\}$ and $e\in\{\sS, \sI, \sR\}$.
\end{itemize}
Figure~\ref{fig:policy_comparison_per_ada} shows that, as expected, the adaptive guideline reduces the proportion of infected individuals across all economic groups compared to the baseline, indicating that a more targeted restrictive policy on infected individuals is effective in lowering overall incidence. Under both policies, the ordering of infection proportion across groups is preserved and follows $\li > \lf \approx \mi > \mf \approx \hi > \hf$, suggesting that lower–income indifferent individuals consistently take the heaviest infection burden while high-income followers have lower infections. Interestingly, we found that being in the follower group partly offsets the disadvantage of being in a lower economic group. 

These results also show that all susceptible individuals reduce their socialization to protect themselves from being infected. In particular, followers reduce their socialization below the recommended level, $(\lambda_t^{k, \sS})_{t\in[0,T]}$ while indifferents stayed below their maximum socialization level $1$. Individuals in the lower–income groups restrict themselves the most at the peaks under both baseline permissive and adaptive guidelines. Strikingly, susceptible followers socialize more under the adaptive guideline, as infected followers sharply reduce their socialization levels. Susceptible individuals then feel better protected and partially \textit{free ride} on this protective behavior of infected follower individuals, engaging in higher socialization: the largest socialization peak difference occurs in $\lf$ (0.0236 above the baseline), while the smallest occurs in $\hi$ (0.0155 above the baseline). Across both policies, vaccination takes the ordered pattern $ \hf > \mf > \lf > \hi > \mi > \li$, which shows that followers choose to vaccinate more than the indifferent individuals as well as higher-income individuals choose to vaccinate more than the lower-income individuals. 
We also emphasize that under the adaptive guideline, the perceived safety from infected follower individuals reduced socialization levels, further decreases the vaccination level decisions of susceptible individuals: the largest vaccination peak difference occurs in $\hf$ (0.0124 below the baseline), while the smallest occurs in $\li$ (0.0075 below the baseline).

Figure~\ref{fig:policy_comparison_per_str} compares the baseline permissive guideline with the strict guideline that recommends stronger non-pharmaceutical restrictions uniformly across all infection states. Under the strict guideline, proportion of infected curves are substantially flattened for every group, with a lower peak infection at a delayed time in comparison to the permissive benchmark. Interestingly, susceptible followers markedly reduce their socialization levels under strict policy due to their desire to follow the more restrictive guidelines. This in turn incentivizes the followers to choose lower vaccination levels to the point that even falls below those of the indifferents, that the vaccination ordering under strict policy becomes $\hi > \mi \approx \li > \hf > \mf \approx \lf$. On imposing the strict guideline, the largest vaccination peak difference occurs in $\hf$ (0.0328 below the baseline), and the smallest occurs in $\li$ (0.0134 below the baseline). This clearly shows a trade-off in the decision-making of susceptible individuals while choosing different preventive behaviors, i.e. social-distancing vs vaccination. On the other hand, we observe that susceptible indifferent individuals enjoy a higher socialization level under the strict guidelines in comparison to baseline permissive guidelines. This again shows a \textit{free-rider} behavior of susceptible indifferent individuals where they can keep their socialization levels closer to the maximum levels due to the strict preventative measures taken by the follower individuals: the socialization peak differences are 0.0340, 0.0322 and 0.0251 for $\li, \mi, \hi$.

We then jointly quantify how these three guidelines shape cross–group disparities. The Low–Indifferent ($\li$) group consistently exhibits the highest infection proportion, and the largest infection group disparity is observed between $\li$ and High–Follower ($\hf$) groups. These $\li$–$\hf$ infection group disparities are 3.859\%, 2.835\%, and 3.853\% under the permissive, adaptive, and strict guidelines respectively. The adaptive guideline most effectively narrows the group disparity in infection proportions across authority perception and economic groups. Moreover, Figure~\ref{fig:peak_comp} shows that as guidelines become stricter, infection peaks are lowered and peak times become more spread out across groups. In particular, the largest infection peak differences are 4.958\% ($\li$, adaptive) and 9.409\% ($\lf$, strict), and the smallest ones are 4.128\% ($\hf$, adaptive) and 6.874\% ($\hi$, strict) from the baseline guideline. The peak time spans are between $\li$ and $\hf$, and valued at 1.9, 2.4 and 4.2 time steps for the permissive, adaptive and strict guidelines. Turning to socialization and vaccination behavior comparisons, the largest $\hf$–$\lf$ socialization group disparities are $0.0283$, $0.0214$, and $0.0198$ under the permissive, adaptive, and strict guidelines, and the respective $\hi$–$\li$ socialization group disparities are $0.0259$, $0.0201$, and $0.0161$. The largest vaccination group disparities are between $\hf$–$\li$ under the permissive (0.0139) and adaptive (0.0091) guidelines, and between $\hi$–$\lf$ under the strict (0.0086) guideline. Overall, as guidelines become stricter, they are more effective in compressing differences in socialization and vaccination behaviors. We also remark that, under our experiment setting, we observe the same trend when peak difference is used to quantify cross-group gaps.

\begin{figure}[h]
    \centering
    \includegraphics[width=0.7\textwidth]{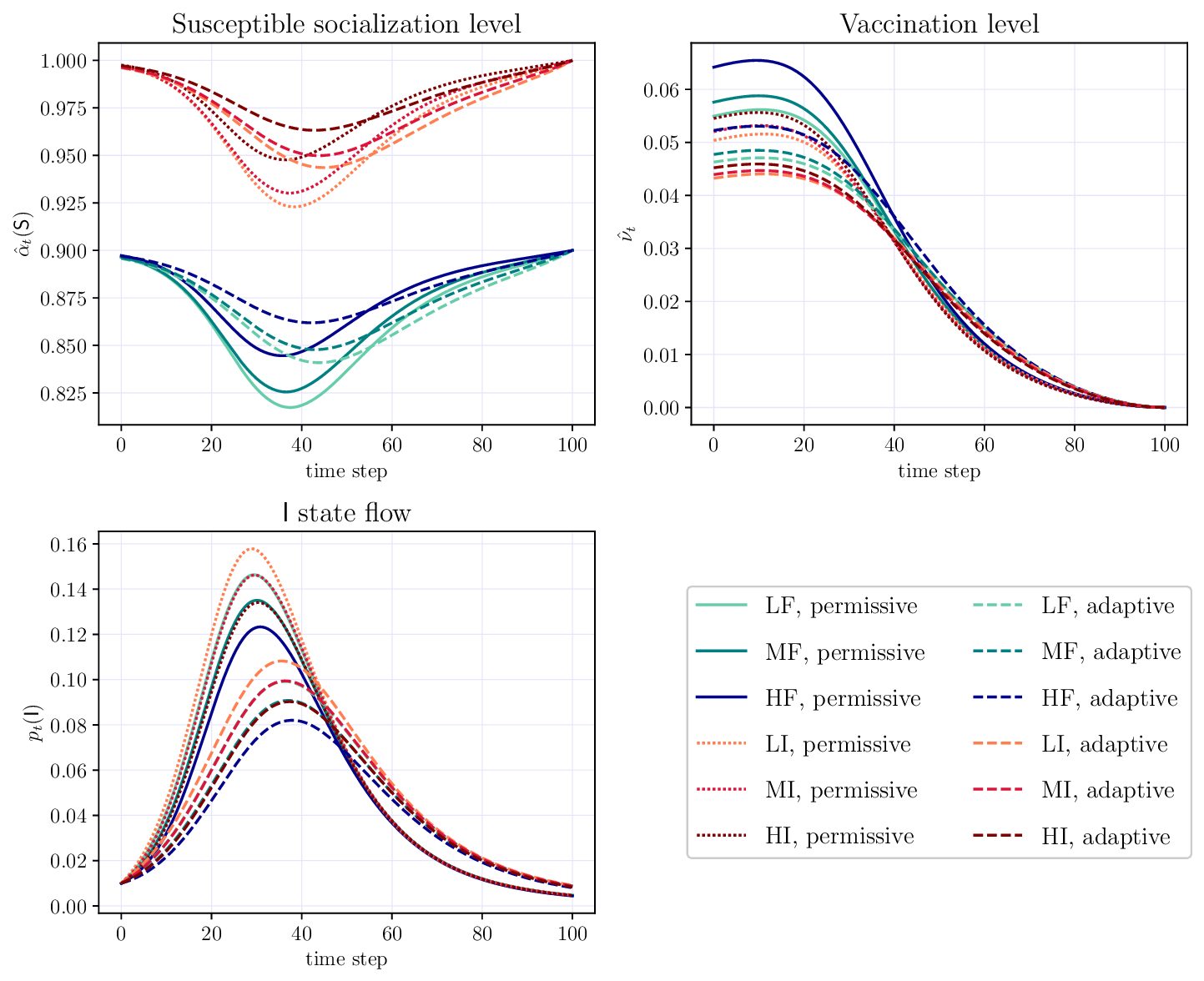}
    \captionsetup{font=footnotesize}
    \caption{
        Comparison of epidemics mitigation effects under two guidelines. \textbf{Permissive}: $\lambda_t^{k, \sS} = 0.9, \lambda_t^{k, \sI} = 0.9, \lambda_t^{k, \sR} = 0.9$ where the results are shown with solid lines for follower groups and with dotted lines for indifferent groups to increase the readability; \textbf{Adaptive}: $\lambda_t^{k, \sS} = 0.9, \lambda_t^{k, \sI} = 0.6, \lambda_t^{k, \sR} = 0.9$, for all $t \in [0,T]$, $k \in \{\lf, \mf, \hf, \li, \mi, \hi\}$ where the results are shown with dashed lines for all groups.}
    \label{fig:policy_comparison_per_ada}
\end{figure}

\begin{figure}[h]
    \centering
    \includegraphics[width=0.7\textwidth]{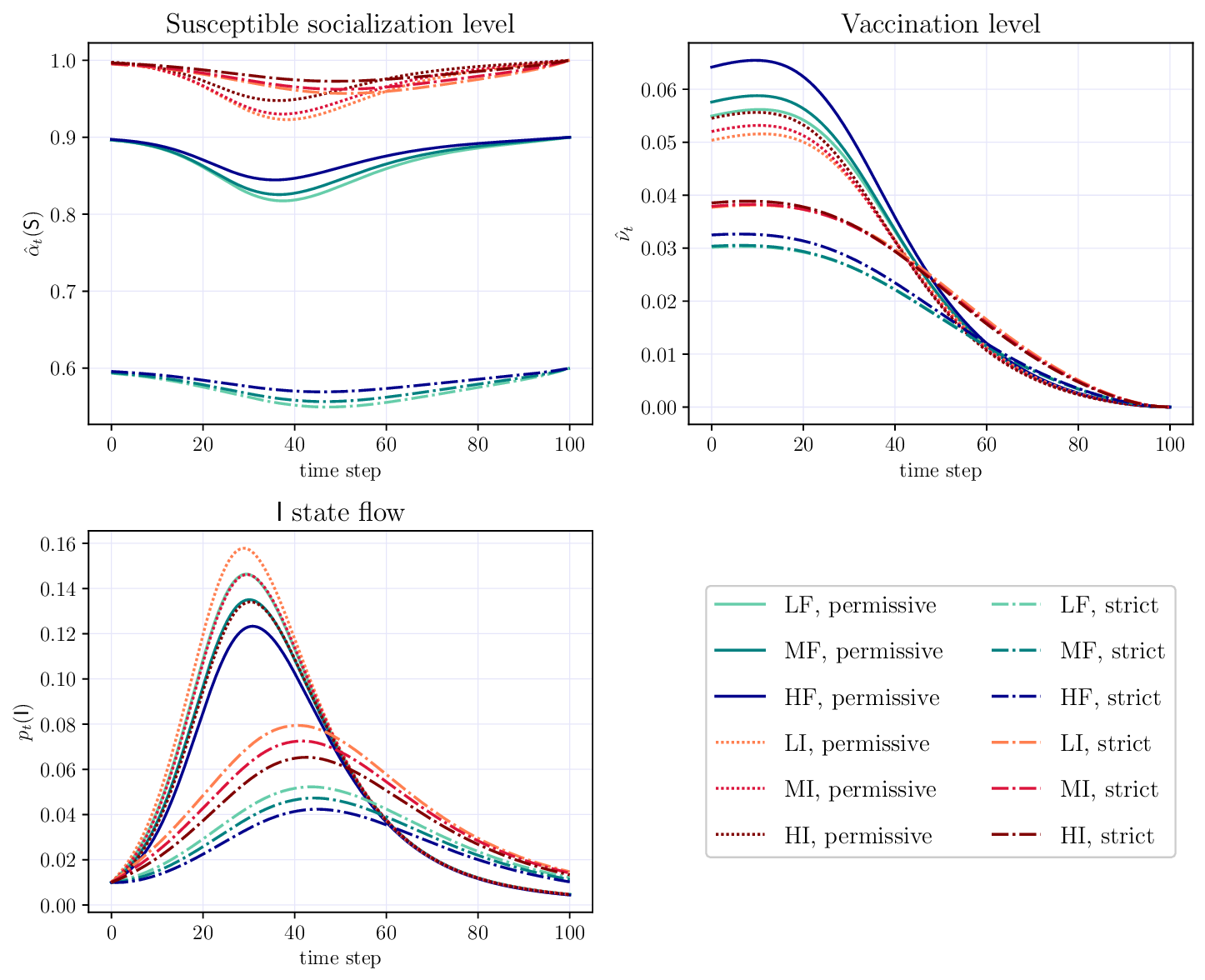}
    \captionsetup{font=footnotesize}
    \caption{
        Comparison of epidemics mitigation effects under two guidelines. \textbf{Permissive}: $\lambda_t^{k, \sS} = 0.9, \lambda_t^{k, \sI} = 0.9, \lambda_t^{k, \sR} = 0.9$ where the results are shown with solid lines for follower groups and with dotted lines for indifferent groups to increase the readability; \textbf{Strict}: $\lambda_t^{k, \sS} = 0.6, \lambda_t^{k, \sI} = 0.6, \lambda_t^{k, \sR} = 0.6$, for all $t \in [0,T]$, $k \in \{\lf, \mf, \hf, \li, \mi, \hi\}$ where the results are shown with dotted-dashed lines for all groups.}
    \label{fig:policy_comparison_per_str}
\end{figure}

\begin{figure}[h]
    \centering
    \includegraphics[width=0.7\textwidth]{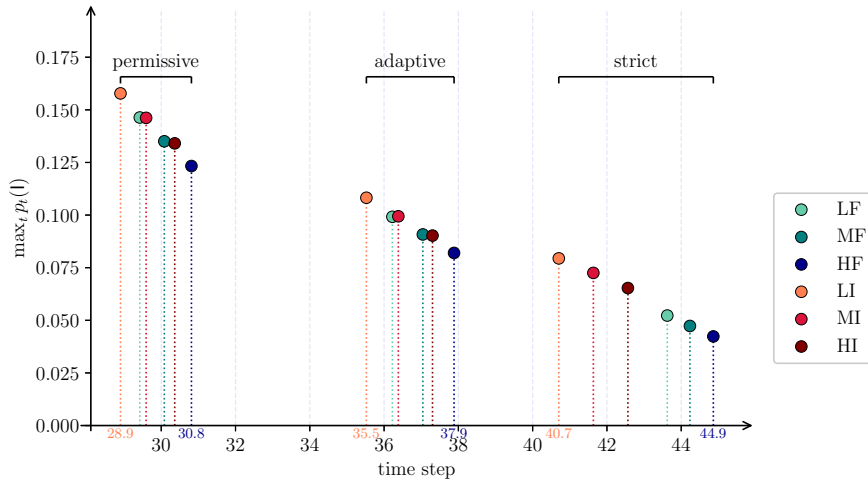}
    \captionsetup{font=footnotesize}
    \caption{Joint comparison of infection peak times and infection peak proportions across guidelines and groups.}
    \label{fig:peak_comp}
\end{figure}

\subsubsection*{Effects of Vaccination Costs}
In Figure~\ref{fig:vax_cost_comparison}, we compare the group-adaptive vaccination cost regime (where $c^{\lf}_{\nu} = 1.4, c^{\mf}_{\nu} = 1.2, c^{\hf}_{\nu} = 0.8, c^{\li}_{\nu} = 1.6, c^{\mi}_{\nu} = 1.4, c^{\hi}_{\nu} = 1.0$) with a uniformly low vaccination cost regime (where $c^{k}_{\nu} =0.8$, for all $k\in[K]$). The group-adaptive parameters are estimated from survey data and represent the increased cost of vaccination due to the lower economic status of the individuals as well as the differences in the perceived risk of vaccination between the authority follower and indifferent individuals. We find that overall while lowering vaccination costs increases vaccination levels across most groups (except for $\hf$, as shown in the difference between the continuous and dashed lines), its impact on infection reduction is limited compared to even milder social-distancing guidelines. Consequently, interventions targeting only reducing vaccination costs may operate at much greater expense to achieve desired outcomes unless complemented by socialization guidelines. 
    
The infection peak differences between the two vaccination cost regimes are 0.791\%, 0.644\%, and 0.375\% for $\lf$, $\mf$, and $\hf$, and 0.922\%, 0.775\%, and 0.517\% for $\li$, $\mi$, and $\hi$. Higher-income groups benefit less from the uniformly low-cost regime since it provides the smallest decrease in their original vaccination cost. We also see the infection peak is suppressed more among indifferent individuals. The largest infection group disparity is between $\li$ (higher) and $\hf$ (lower) under both regimes, at 3.859\% (adaptive) and 3.372\% (low-cost) which shows that the low vaccination cost regime decreases the group disparities. Vaccination levels of individuals are substantially reordered under the low cost regime, where we observe the following ordering. Originally the vaccination ordering was as $\hf > \mf > \lf \approx \hi >\mi>\li$ which shows higher vaccination levels for the followers and high-economic status individuals. In the low cost regime, the ordering becomes $\li > \lf > \mi > \mf > \hi > \hf$, with mid- and low-income groups increasing their vaccination levels the most by taking taking advantage of lowered costs. We find that indifferents in the new regime perceive a lower vaccination cost and have higher vaccination levels than followers within the same economic group. Only individuals in $\hf$ group decrease their vaccination levels, and this is due to fact that they can take advantage of the increased vaccination levels in the other groups. Specifically, the largest vaccination peak difference occurs in $\li$ (0.0489 above the adaptive regime) and the smallest one occurs in $\hf$ (0.0012 below the adaptive regime). In both regimes, the largest vaccination group disparity is between $\hf$ and $\li$, but $\hf$ vaccinates more under adaptive costs and less under low costs, where the group disparities are $0.0139$ and $0.0366$ respectively. Overall, a uniformly low vaccination cost narrows infection group disparities but increases the vaccination group disparities. Due to higher vaccination levels, individuals in all groups can slightly increase their socialization levels.

\begin{figure}[h]
    \centering
    \includegraphics[width=0.7\textwidth]{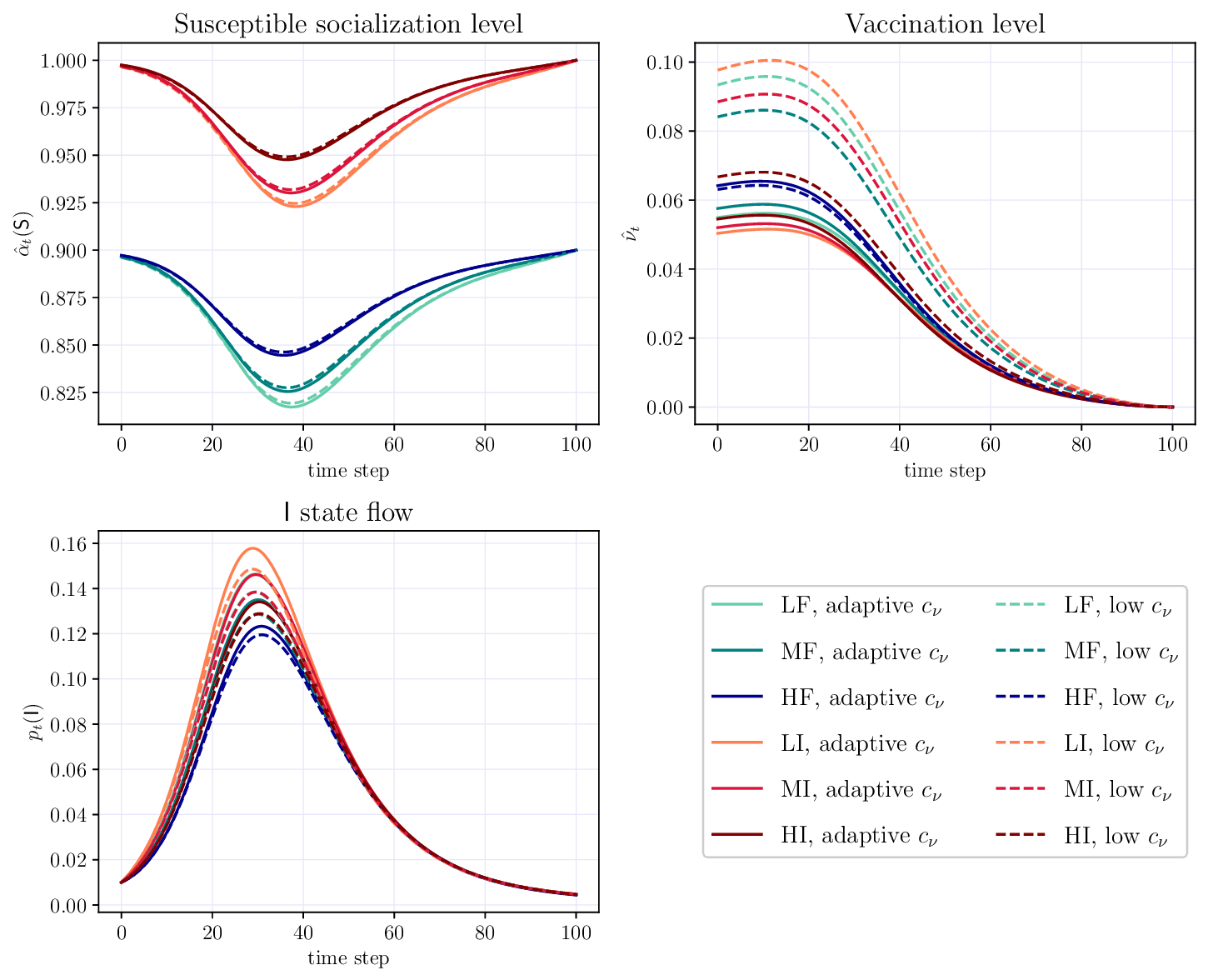}
    \captionsetup{font=footnotesize}
    \caption{
        Comparison of epidemics mitigation effects under two vaccination cost regimes. \textbf{Adaptive vaccination cost}: $
    c^{\text{LF}}_{\nu} = 1.4, c^{\text{MF}}_{\nu} = 1.2, c^{\text{HF}}_{\nu} = 0.8, c^{\text{LI}}_{\nu} = 1.6, c^{\text{MI}}_{\nu} = 1.4, c^{\text{HI}}_{\nu} = 1.0$ where the results are shown with solid lines for all groups; \textbf{Low vaccination cost}: $c^{\text{LF}}_{\nu} = c^{\text{MF}}_{\nu} = c^{\text{HF}}_{\nu} = c^{\text{LI}}_{\nu} =  c^{\text{MI}}_{\nu} = c^{\text{HI}}_{\nu} = 0.8$ where the results are shown with dashed lines for all groups.}
    \label{fig:vax_cost_comparison}
\end{figure}

\subsubsection*{Effects of Authority Perception} 
Figure~\ref{fig:mixed_followers_comparison} compares two population structures under the permissive social-distancing guideline. In the mixed-perception population, about $52\%$ of individuals are authority indifferent individuals; in the all-follower population, everyone is an authority follower and individuals differ only by economic status. For parameters that originally vary by authority perception in the mixed-perception case (network connections and $c_{\nu}$), we use their averages when constructing the all-follower regime. 

Under the all-follower regime, infection spread is markedly reduced: infection peak differences are 2.323\%, 2.179\% and 2.012\% for $\lf$, $\mf$ and $\hf$, with low-income groups benefiting the most. In this environment, follower individuals enjoy higher socialization levels and lower vaccination levels given a stronger population-level prevention: the socialization peak differences are 0.0121, 0.0114, 0.0088 for $\lf$, $\mf$, $\hf$; and the vaccination peak differences are 0.0067, 0.0079, 0.0112 respectively. Authority perception thus emerges as a dominant factor in epidemic mitigation: populations composed entirely of authority followers exhibit strong mitigation even under permissive social-distancing guidelines, whereas mixed-perception populations experience higher infection rates under the same social distancing guidelines. This finding underscores the importance of trust-building measures—such as transparent communication, community engagement, and consistent messaging—in epidemic preparedness. 

The largest infection and vaccination group disparities are between $\li$ and $\hf$ in the mixed-perception population ($3.859\%$ and $0.0139$), and between $\mathsf{L}(\mathsf{F})$ and $\mathsf{H}(\mathsf{F})$ in the all-follower regime ($2.228\%$ and $0.0048$). Respectively, socialization group disparities between low- and high-income followers are $0.0282$ and $0.0247$. Overall, the all-follower regime narrows group disparities in infection, vaccination, and socialization outcomes.

\begin{figure}[h]
    \centering
    \includegraphics[width=0.7\textwidth]{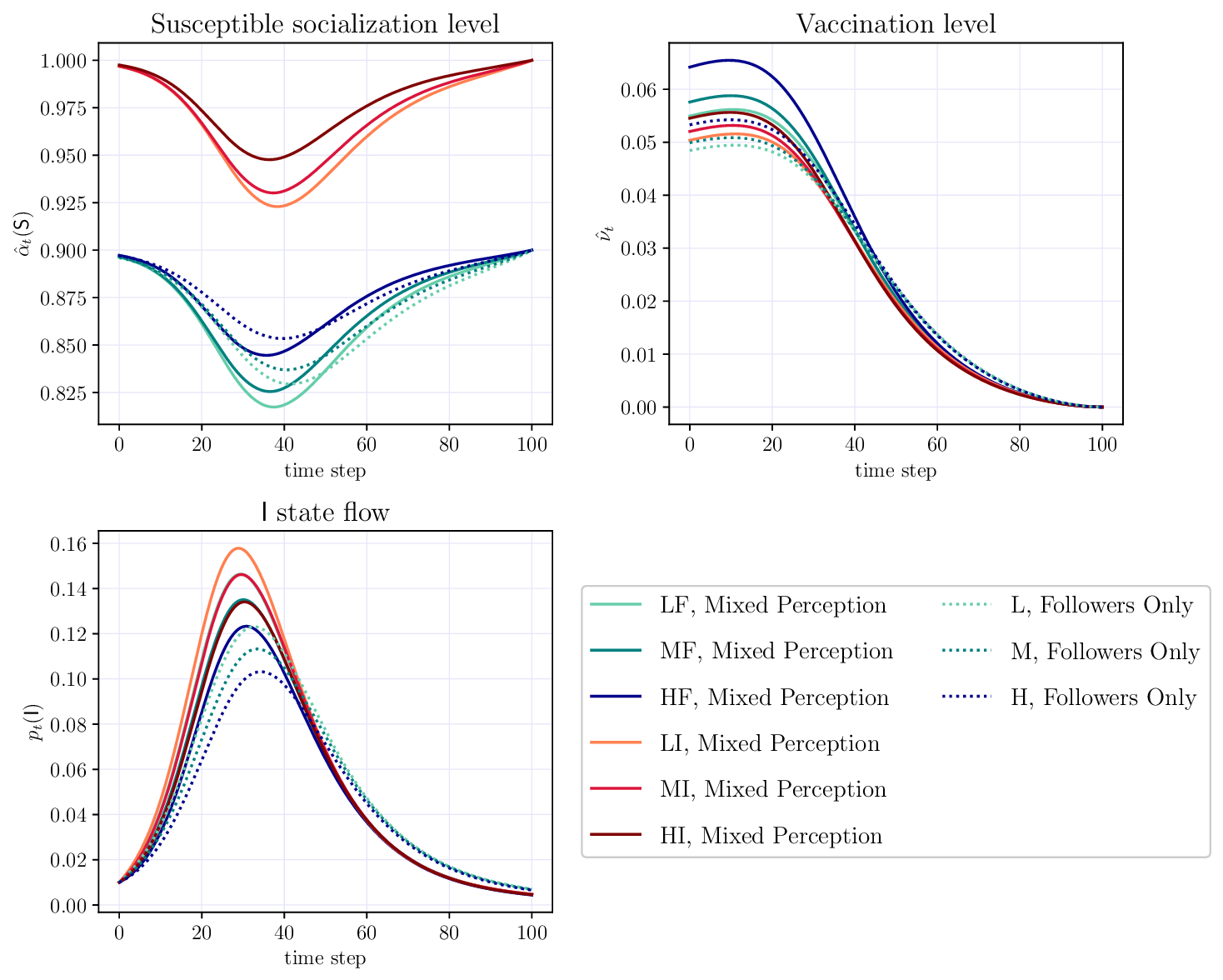}
    \captionsetup{font=footnotesize}
    \caption{
        Comparison of epidemics mitigation effects under two population structure regimes. \textbf{Mixed perception}: population is stratified with authority perceptions and economic status (i.e. 6 groups: $\lf$, $\mf$, $\hf$, $\li$, $\mi$, $\hi$); \textbf{Followers only}: population is only stratified with economic status (i.e. 3 groups: $\mathsf{M}$($\mathsf{F}$), $\mathsf{M}$($\mathsf{F}$), $\mathsf{H}$($\mathsf{F}$)).
    }
    \label{fig:mixed_followers_comparison}
\end{figure}

\subsection*{Discussion}
Our findings reveal that epidemic mitigation is not solely a matter of imposing stricter social-distancing policies or reducing (perceived) vaccination costs; rather, it depends critically on behavioral response of individuals and their heterogeneity shaped by economic status and authority perception. Using a survey‑informed, multi‑population mean field game with vaccination and socialization level controls, we find that targeting social distancing to infected individuals consistently lowers infection proportions across income and authority‑perception groups, with susceptible followers partially increasing their socialization as infected followers sharply reduce their socialization levels. Stricter guidelines further flatten and delay infection peaks for all groups, but they also depress vaccination among follower individuals to the point that vaccination levels can fall below those of indifferent individuals, revealing a trade-off between non-pharmaceutical distancing and vaccination choices.

These dynamics are governed by two mechanisms. First of them is the \textit{behavioral trade-off}: when guidelines tighten, susceptible individuals perceive greater protection due to others’ reduced socialization levels and they lower their vaccination. Second is the \textit{free‑riding}: susceptible indifferent individuals raise socialization levels under stricter guidelines, with measurable socialization level peak increases for $\li$, $\mi$, and $\hi$ relative to the results under permissive guidelines.

Policy implications follow directly from these findings. Adaptive guidelines that focus on infected individuals are effective at reducing the infection peaks without negatively affecting the socialization levels of susceptible individuals. Conversely, lowering vaccination costs alone raises vaccination levels especially in mid‑ and low‑income groups, yet has a modest effect on decreasing infection peaks, suggesting that cost‑only interventions may be insufficient unless paired with socialization guidelines on infected individuals. 
We observe that low‑income indifferent individuals bear the highest infection burden across different regimes, and the largest infection disparity typically arises between $\li$ and $\hf$ groups. Stricter guidelines compress disparities in socialization and vaccination levels, while uniformly low vaccination costs narrow infection disparities and incentivize low- and mid-income groups to vaccinate more than high-income groups. These findings suggests intervention strategies that integrate (i) infected‑targeted distancing (e.g., isolation/quarantine support, mask guidance), (ii) trust‑building (transparent communication, community engagement to increase desire of individuals to follow the guidelines), and (iii) targeted economic support to reduce logistical and perceived vaccination costs for low‑income groups.

This work complements and extends prior modeling paradigms. Unlike compartmental models that treat interventions as exogenous modifiers of transmission rates, our framework endogenizes human response and computes the equilibrium behavior induced by interacting decision‑makers, yielding a forward–backward differential system that characterizes behavior and disease dynamics jointly. 

The modeling of individual decision-making in our paper also contrasts with the (stochastic or deterministic) optimal control models that focus on optimal public health policy-making to minimize the epidemic spread~\cite{acemoglu2021optimal,jia2025learning,rowthorn2009optimal,lu2021stability} by modeling the decision-making of the public health authority instead of modeling the objectives and decision-making of large number of individuals in the population. These optimal control models assume the infection dynamics in the population follow mechanistic models in which the transition rates are directly affected by the policy-makers decisions. In contrast, in our setting policy guidelines do not directly affect the transmission rates instead they indirectly affect them via the objective functions.

Finally, our model offers higher tractability in comparison to agent-based models and simulations~\cite{abm_kerr2021covasim,abm_kloh2020virus,abm_perez2009agent,abm_silva2020covid}. Agent-based models offer modeling of heterogeneity and individual decision making, but can be computationally intensive and opaque for policy optimization. In comparison to these models, our study characterizes Nash equilibrium with a differential equation system, guarantees the existence and uniqueness of equilibrium behavior, and proposes a computationally efficient numerical algorithm. Recent feedback-informed models~\cite{decision-making-sim} embed economic choice into epidemic dynamics, echoing our finding that behavioral feedback reshapes policy effectiveness. Our mean-field framework provides a mathematically tractable lens on these dynamics. As we mentioned in the introduction, MFGs and their extensions have been used to model the epidemic spread among rational individuals~\cite{aurell2022optimal,aurell2022finite,bremaud2024mean,olmez2022modeling,lee2020controlling,elie2020contact,cho2020mean,Doncel_Gast_Gaujal_2022,pnas_epidemics}. However, none of these studies have considered in their models (1) the joint rational decision-making of individuals concerning both vaccination and social distancing, and (2) population  stratification based on economic status and perceptions of authority.

The limitations of our study is as follows. We model individuals as homogeneous within each of six groups with interactions among individuals captured through mean‑field terms in an infinite‑population limit. Real populations exhibit richer within‑group heterogeneity and network effects. In addition, the empirical calibration relies on a two‑wave Illinois survey conducted in 2024–2025 and therefore, external validity beyond this time frame requires further investigation. 

By integrating survey-informed parameters into a rigorous and tractable game-theoretic framework, this study bridges theoretical modeling with observed behavioral patterns to understand the human response to policy-making related to epidemic control. Future work could extend this approach to study \textit{optimal} policies for public health authorities under different objectives such as diminishing health disparities, flattening the infected curve or minimizing the economical impacts. Ultimately, our approach provides a interdisciplinary strategy that combines epidemiological modeling, game-theoretical tools, applied mathematics, and social trust to design interventions that are both effective and equitable.

\bibliographystyle{plain}
\bibliography{references}

\appendix
\section*{Supporting Information}

\section{From SIR to Mean Field Game State Transitions}
\label{sec:sir_state dynamics_explanation}
To further explain how we formulate the state transitions, we begin with a finite-population model with $N$ individuals. At time $t\in[0,T]$, each individual $j\in\{1,\ldots,N\}$ occupies a health state $X_t^{j,N}\in\{\sS,\sI,\sR\}$. The empirical proportion in each compartment is
$$
p_t^N(e)=\frac{1}{N}\sum_{j=1}^N \mathbf{1}_e(X_t^{j,N}), 
\quad e\in\{\sS,\sI,\sR\},
$$
where $\mathbf{1}_e(X_t^{j,N})$ is the indicator function that is equal to 1 if $X_t^{j,N}=e$ and equal to $0$ otherwise. Inspired by the SIR model, we start by modeling the \textit{uncontrolled} state transitions for individuals without the vaccination or social distancing decisions. For the sake of simplicity in notations and explanations, we also ignore waning of immunity in this first step. In this model, individuals meet with each other pairwise at a random rate. The number of encounters with infected individuals at time $t$ is proportional
to the the proportion of infected individuals at that time, $p_t^N(\sI)$ and a base transmission rate $\beta$. Therefore, a susceptible individual becomes infected at rate $\beta\, p_t^N(\sI)$ (i.e., $\sS\to\sI$) and an infected individual recovers at a constant rate $\gamma$ (i.e., $\sI\to\sR$), where $\beta,\gamma>0$. The state process of the individual accordingly follows a continuous time Markov chain with the following transition-rate matrix:
\begin{equation}
\label{eq:q_matrix_sir_uncontrolled}
\mathcal{T}(p_t^N)=
\begin{blockarray}{ccccc}
& \sS & \sI & \sR \\[0.5mm]
\begin{block}{c[cccc]}
  \sS & \cdots  & \beta p_t^N(\sI) & 0 \\
  \sI & 0  & \cdots & \gamma \\
  \sR & 0 & 0 & \cdots \\
\end{block}
\end{blockarray}.
\end{equation}
To emphasize the connection of the above individualistic setting, we state that following the classical results from~\cite{kurtz1981approximation}, when the number of individuals is taken to infinity the limit of $p_t^N$ becomes a solution to the following equations:
\begin{equation*}
\begin{aligned}
    \frac{dp_t(\sS)}{dt} &= -\beta p_t(\sI) p_t(\sS)\\
    \frac{dp_t(\sI)}{dt} &= \beta p_t(\sI) p_t(\sS)-\gamma p_t(\sI)\\
    \frac{dp_t(\sR)}{dt} &= \gamma p_t(\sI),\\
\end{aligned}
\end{equation*}
with initial conditions $p_0^N$ are given.
We emphasize that since the number of susceptible, infected, and recovered individuals are given as $(S, I, R)=(Np_t(\sS), N p_t(\sI), Np_t(\sR))$, the above equation system becomes the classical SIR model:
\begin{equation*}
\begin{aligned}
    \frac{dS}{dt} &= -\frac{\beta}{N} I S\\
    \frac{dI}{dt} &= \frac{\beta}{N} IS -\gamma I\\
    \frac{dR}{dt} &= \gamma R,\\
\end{aligned}
\end{equation*}
with initial number of individuals in each compartment $(S_0, I_0, R_0)$ are given. To incorporate the individual decision making, extending the formulation in~\cite{aurell2022optimal}, we allow individuals to choose their social activity (i.e, socialization level) and also vaccination. Let $\alpha_t^{j,N}\in[0,1]$ denote the socialization level of individual $j$ at time $t$. Encounters remain random, but infection now depends on individual's own socialization level and also others' infection status (i.e., whether they are infected or not) and socialization levels. Therefore, a susceptible individual $j$ becomes infected at rate
$$
\beta\,\alpha_t^{j,N}\Big(\frac{1}{N}\sum_{i=1}^N \alpha_t^{i,N}\,\mathbf{1}_{\sI}(X_t^{i,N})\Big).
$$
As $N\to\infty$, by using law of large numbers, the interactions occur through population distributions instead of empirical distributions. Denoting the joint distribution of control and state by $\pi_t$, a representative susceptible individual using action $\alpha_t$ transitions to state $\sI$ at the limiting rate
$$
\beta\,\alpha_t \int_A a\,\pi_t(da,\sI),
$$
which summarizes the infection pressure generated by the population's collective behavior. We also allow vaccination: individual $j$ chooses a vaccination intensity $\nu_t^{j,N}\ge 0$, which induces a transition $\sS\to\sR$ at a rate proportional to $\nu_t^{j,N}$.
Finally, to capture heterogeneity, we partition the population into $K$ many groups (in this work, by economic status and authority perception) with possibly different model parameters. The group sizes are denoted $N^1, \ldots, N^k$ and they satisfy $N^1+\cdots+N^K=N$. The proportion of group $k$ is given by $m^k:=N^k/N$.
In the limiting regime where $N$ goes to infinity, proportions of groups are assumed to remain fixed. Under large population limit, this yields a controlled mean-field model in which a representative agent (within each group) optimizes her socialization and vaccination levels in response to the aggregate flow generated by the population.

\section{SIR-type Model Theoretical Results}
\label{sec:sir_theoretical_result}

First, we give the characterization of the multi-population mean field game (MFG) Nash equilibrium with forward-backward ordinary differential equations (FBODE). Intuitively, this means, finding the multi-population MFG Nash equilibrium will be equivalent to solving the FBODE given in Theorem~\ref{the:fbode_sir}.

\begin{assumption}
\label{assu:fbode}
    Public health policy functions $\lambda^{k,e}: [0,T]\to[0, \bar{\lambda}]$ for all $k\in[K]$ and $e\in\{\sS, \sI, \sR\}$ are continuous.
\end{assumption}

\begin{theorem}
\label{the:fbode_sir}
    Under assumption~\ref{assu:fbode}, multi-population mean field game Nash equilibrium controls are given as 
    \begin{align}
        \hat{\alpha}_t^k(\sS) &= \lambda_t^{k, \sS} + \frac{\beta^k Z_t^k \big(u_t^k(\sS) - u_t^k(\sI)\big)}{2c_\lambda^k} \qquad\qquad&& \hat{\alpha}_t^k(\sS) = 1 + \frac{\beta^k Z_t^k \big(u_t^k(\sS) - u_t^k(\sI)\big)}{2c_\lambda^k} \label{eq:alphaS}\\
        \hat{\alpha}_t^{k}(\sI) &= \lambda_t^{k, \sI} && \hat{\alpha}_t^{k}(\sI) =\xi_{\sI}\label{eq:alphaI}\\
        \hat{\alpha}_t^{k}(\sR) &= \lambda_t^{k, \sR}\label{eq:alphaR} &&\hat{\alpha}_t^{k}(\sR)=1\\
        \hat{\nu}^k_t &= \underbrace{\frac{\kappa^k \big(u_t^k(\sS) - u_t^k(\sR)\big)}{2c_\nu^k}}_{{k\:\in\:\{\lf,\: \mf,\: \hf}\}} && \hat{\nu}^k_t= \underbrace{\frac{\kappa^k \big(u_t^k(\sS) - u_t^k(\sR)\big)}{2c_\nu^k}}_{{k\:\in\:\{{\li,\: \mi,\: \hi}}\}}\label{eq:Nu}
    \end{align}
    for all $k \in[K]$ if $(\bp,\bu)=(\bp^k, \bu^k)_{k\in[K]}$ tuple solves the following forward-backward ordinary differential equation (FBODE) system:    
\begin{equation}
\label{eq:FBODE}
\begin{aligned}
    \dot p^k_t(\sS) &= -\beta^k \hat{\alpha}^{k}_t(\sS) Z_t^k p_t^k(\sS) - \kappa^k \hat{\nu}_t^kp_t^k(\sS) + \eta^k p_t^k(\sR) ,\\[1mm]
    \dot p^k_t(\sI) &= \beta^k \hat{\alpha}^{k}_t(\sS) Z_t^k p_t^k(\sS) - \gamma^k p_t^k(\sI),\\[1mm]
    \dot p^k_t(\sR) &= \gamma^k p_t^k(\sI) +\kappa^k \hat{\nu}_t^kp_t^k(\sS) - \eta^k p_t^k(\sR),
    \\[1mm]
    \dot u^{k_1}_t (\sS)&= \beta^{k_1} \hat\alpha^{k_1}_t(\sS)Z^{k_1}_t\big(u^{k_1}_t(\sS) - u^{k_1}_t( \sI)\big) + \kappa^{k_1} \hat{\nu}_t^{k_1} \big(u^{k_1}_t(\sS) - u^{k_1}_t( \sR)\big) 
    -c^{k_1}_\lambda \big(\lambda^{{k_1}, \sS}_t-\hat\alpha^{k_1}_t(\sS)\big)^2 - c_\nu^{k_1} (\hat{\nu}_t^{k_1})^2,
    \\[1mm]
    \dot u^{k_2}_t (\sS)&= \beta^{k_2} \hat\alpha^{k_2}_t(\sS)Z^{k_2}_t\big(u^{k_2}_t(\sS) - u^{k_2}_t( \sI)\big) + \kappa^{k_2} \hat{\nu}_t^{k_2} \big(u^{k_2}_t(\sS) - u^{k_2}_t( \sR)\big) 
    -c^{k_2}_\lambda \big(1-\hat\alpha^{k_2}_t(\sS)\big)^2 - c_\nu^{k_2} (\hat{\nu}_t^{k_2})^2,
    \\[1mm]
    \dot u^k_t (\sI)&= \gamma^k \big(u^k_t(\sI) - u^k_t( \sR)\big) - c^k_{\sI},
    \\[1mm]
    \dot u^k_t (\sR)&= \eta^k (u_t^k(\sR)-u_t^k(\sS)),
    \\[1mm]
     u^k_T(e)&= 0,\quad p^k_0(e) = p^k(e),\quad \forall e\in \{\sS, \sI, \sR\},  \\[1mm]
     Z^k_t &= \sum_{l \in [K]} w(k, l) \EE[\hat{\alpha}_t^l(\sI)] p_t^l(\sI) m^l, \quad \forall k \in [K],\: \forall k_1\in\{\lf,\mf, \hf\},\: \forall k_2\in\{\li, \mi, \hi\},\: \forall t\in[0,T].
\end{aligned}
\end{equation}
\end{theorem}

The forward dynamics with the initial condition (i.e., dynamics of $(p_t^{k}(e))_{k\in[K], t\in[0,T], e\in\{\sS, \sI, \sR\}}$) in the FBODE system~\eqref{eq:FBODE} represent how the proportion of susceptible, infected, and recovered individuals evolves in groups when population is at the multi-population mean field Nash equilibrium. We use notation $u_t^k(e)$ for all $e\in \{\sS, \sI, \sR\}$ to represent the value function of the representative agent in group $k$ at time $t$ in state $e$ which is defined for $k\in\{\lf, \mf,\hf\}$ as:
\begin{equation*}
\begin{aligned}
    &u_t^k(e) :=  
\inf_{(\alpha^k_\tau, \nu^k_\tau)_{\tau\in[t,T]}} \ \mathbb{E}\Big[\int_t^T \Big[\big(c^k_{\lambda} (\lambda_\tau^{k,\sS} -\alpha_\tau^k)^2 + c_{\nu}^k (\nu_\tau^k)^2 \big)\mathbf{1}_{\sS}(X^k_\tau) \\&\hskip5cm+ 
\big((\lambda_\tau^{k,\sI} -\alpha_\tau^k)^2 + c_{\sI}^k\big)\mathbf{1}_{\sI}(X^k_\tau) 
    +(\lambda_\tau^{k,\sR} -\alpha_\tau^k)^2\mathbf{1}_{\sR}(X^k_\tau)\Big] d\tau\Big| X_t = e\Big],
    \end{aligned}
\end{equation*}
and for $k\in\{\li, \mi,\hi\}$ as:
\begin{equation*}
\begin{aligned}
    &u_t^k(e) :=  
\inf_{(\alpha^k_\tau, \nu^k_\tau)_{\tau\in[t,T]}} \ \mathbb{E}\Big[\int_t^T \Big[\big(c^k_{\lambda} (1 -\alpha_\tau^k)^2 + c_{\nu}^k (\nu_\tau^k)^2 \big)\mathbf{1}_{\sS}(X^k_\tau) \\&\hskip5cm+ 
\big((\xi_{\sI} -\alpha_\tau^k)^2 + c_{\sI}^k\big)\mathbf{1}_{\sI}(X^k_\tau) 
    +(1 -\alpha_\tau^k)^2\mathbf{1}_{\sR}(X^k_\tau)\Big] d\tau\Big| X_t = e\Big].
    \end{aligned}
\end{equation*}
Intuitively, the value function $u_t^k(e)$ gives the expected minimum cost between time $t$ and time horizon $T$ when the representative agent (in group) $k$ starts from state $e$ and state transitions follow the controlled dynamics given in Figure 1 in the main text. Then, the backward equations in the FBODE~\eqref{eq:FBODE} represent the dynamics of the value functions of representative agent in group $k$ in the states susceptible, infected, and recovered at the multi-population mean field Nash equilibrium. 

In FBODE~\eqref{eq:FBODE}, the forward and backward components are coupled through $\hat{\alpha}_t^k(\sS)$ and $\hat{\nu}^k_t$. Furthermore, the above system includes the equations for all $k\in[K]$ which are all coupled with each other through the weighted mean field interactions (i.e. aggregate socialization level of infected individuals), $Z_t^k$, which creates further challenges in solving them.

\begin{proof}[Proof of Theorem~\ref{the:fbode_sir}]
    We first describe the transitions between infection states of representative individual in group $k \in [K]$ by a transition rate matrix of the continuous time Markov chain
\begin{equation}
\label{eq:q_matrix_sir}
\begin{blockarray}{ccccc}
& \sS & \sI & \sR \\[0.5mm]
\begin{block}{c[cccc]}
  \sS & \cdots  & \beta^k \alpha_{t}^{k} Z_{t}^{k} & \kappa^k \nu_t^k \\
  \sI & 0  & \cdots & \gamma^k \\
  \sR & \eta^k & 0 & \cdots \\
\end{block}
\end{blockarray}
\end{equation}
    where notation $\dots$ on each row represents the negative of the sum of all other entries in the same row to make the row sum of the transition rate matrix of the continuous-time Markov chain equal to $0$. We further denote by $q^k_t(e, e^{\prime}, z, \alpha, \nu)$ the transition rate from state $e$ to $e^\prime$, where $e, e' \in\{\sS, \sI, \sR\}$, for the representative individual. It corresponds to the entry of matrix~\eqref{eq:q_matrix_sir} indexed by $(e, e')$. The Hamiltonian of the representative individual in group $k\in[K]$ could then be written as follows:
    \begin{equation*}
    \begin{aligned}
        H^k(t, e, z, \alpha, \nu, u)=\sum_{e^\prime\in \{\sS, \sI, \sR\}} q^k_t(e, e^{\prime}, z, \alpha, \nu)u(e^{\prime}) + f^k(t, e, z, \alpha, \nu), 
    \end{aligned}
    \end{equation*}
    where $$f^k(t,e,z,\alpha,\nu) = \big(c^k_{\lambda} (\lambda_t^{k,\sS} -\alpha)^2 + c_{\nu}^k \nu^2 \big)\mathbf{1}_{\sS}(e) + \big((\lambda_t^{k,\sI} -\alpha)^2 + c_{\sI}^k\big)\mathbf{1}_{\sI}(e)+ (\lambda_t^{k,\sR} -\alpha)^2\mathbf{1}_{\sR}(e).$$ By dynamic programming principle of optimal control, the Nash equilibrium controls will be derived by minimizing the Hamiltonian with respect to the controls. Since $f^k$ is strongly convex and additive separable on $(\alpha, \nu)$, the mapping $[0,1] \times [0, V] \ni (\alpha,\nu) \mapsto H^k(t, e, z, \alpha, \nu)$ admits a unique measurable and explicitly solvable minimizer. Denote the minimizer by $\hat{\alpha}^k_t(e)$ for the socialization level and $\hat{\nu}_t^k$ for the vaccination level, we compute it using first-order conditions to achieve~\eqref{eq:alphaS},~\eqref{eq:alphaI},~\eqref{eq:alphaR} and~\eqref{eq:Nu}.
    This minimizer will be functions of the other input variables of the Hamiltonian function: $t, z, u$. Then the Hamilton-Jacobi-Bellman (HJB) equation for the finite-state model can be written as:
    \begin{equation*}
        \dot{u}_t^k(e) = - H^{k}(t, e, {Z}_t^k, u^k_t(\cdot), \hat{\alpha}_t^k(\cdot), \hat{\nu}^k_t), \ u_T^{k}(e) = 0, \ t\in[0,T], e\in E.
    \end{equation*}
    where $Z_t^k = \sum_{l \in [K]} w(k, l) \EE[\hat{\alpha}_t^l(\sI)] p_t^l(\sI) m^l $ denotes the aggregate that is calculated by using the MFG Nash equilibrium controls. At the MFG Nash equilibrium, this HJB system will be coupled with the following Kolmogorov-Fokker-Planck (KFP) equation that gives the dynamics of the state distributions at the equilibrium:
    \begin{equation*}
        \dot{p}_t^k(e) = \sum_{e^{\prime} \in E} q^k_t(e^{\prime}, e, Z_t^k, \hat{\alpha}_t^k, \hat{\nu}_t^k) p_t^k(e^{\prime}), \ p_0^{k}(e) = p^k(e), \ t \in [0,T], e \in E.
    \end{equation*}
        Since $\hat{\alpha}_t^k$ and $\hat{\nu}_t^k$ are functions of $Z_t^k$ and $u_t^k$, the (forward) KFP and (backward) HJB equations are coupled to achieve~\eqref{eq:FBODE}.
\end{proof}

\begin{theorem}
\label{the:mfg_nash_existence_uniqueness_SIR}
    Under short-time conditions, there exists a unique bounded solution $(\bp,\bu)$ to the FBODE system in~\eqref{assu:fbode}. 
\end{theorem}

\begin{remark}[Numerical benefit]
    Theorem~\ref{the:mfg_nash_existence_uniqueness_SIR} states that when the time horizon $T$ is small enough, there exists a unique multi-population MFG Nash equilibrium in the SIR epidemic model of individuals with different economic status and authority perception. Numerically, this unique equilibrium can be obtained via a fixed-point iteration (namely small-time solver), and it could be extended to larger horizons by patching consecutive runs~\cite{chassagneux2019numerical}. When $T$ is large, we partition $[0,T]$ into $N$ sub-intervals of length $\tau$, with $\tau$ small enough for the small-time solver to apply. A recursive global solver is then defined backward in time from step $N$ to $1$: for step $k \leq N-1$, the small-time solver is called using the terminal condition returned from step $k + 1$ and the initial condition is known upon recursive call. The procedure ends with terminal condition $u_T = 0$ at base case $k = N$.
\end{remark}

\begin{proof}[Proof of Theorem~\ref{the:mfg_nash_existence_uniqueness_SIR}]
    We prove the unique existence by using Banach Fixed Point Theorem. To accomplish this we construct a fixed point mapping by using the forward and backward components of the FBODE system and show that it is a contraction mapping. Since the equations of authority followers and indifferents differ only by model parameters, we replace $1$ and $\xi_{\sI}$ from the cost function of authority indifferents by using $\lambda_t^{k,\sS}=\lambda_t^{k,\sR}=1$ and $\lambda_t^{k, \sI}=\xi_{\sI}$ for $k\in\{\li, \mi, \hi\}$ and for $t\in[0,T]$ for the simplicity in the presentation of proof. There are in total $K=6$ economic/authority perception groups with $[K]:=\{\lf, \li, \mf, \mi, \hf, \hi\}$, and $3$ states $\sS$, $\sI$ and $\sR$ where $E:=\{\sS, \sI, \sR\}$. We consider the state density flow $p := (p_t^{k}(\sS),p_t^{k}(\sI),p_t^{k}(\sR))_{k \in [K], t \in [0,T]}$ and value function $u := (u_t^{k}(\sS),u_t^{k}(\sI),u_t^{k}(\sR))_{k \in [K], t \in [0,T]}$ as two vector processes. In the construction of the fixed point map by using the forward-backward components of the FBODE, we first fix $p$ flow and solve the backward HJB equations to obtain $u$; and by using these $u$, we solve the forward KFP equations to obtain $\tilde{p}$. We show that this mapping is a contraction mapping for $p \mapsto h_1(p)=u \mapsto h_2(u) = \tilde{p}$, that is, let $h:=h_2 \circ h_1$, for $p^{1} = (p^{1,k}_{t}(e))_{t\in[0,T],k\in[K],e\in E}$ and $p^{2} = (p^{2,k}_{t}(e))_{t\in[0,T],k\in[K],e\in E}$, 
    $$\left\|h\left(p^1\right)-h\left(p^2\right)\right\|_T \leq C \left\|p^1-p^2\right\|_T$$ where $C<1$ and $\left\|p\right\|_{T}:=\sup_{t \in [0,T]}\left\|p_t\right\|^2 = \sup_{t \in [0,T]}\sum_{k\in[K],e\in E} \left(p_t^{k}(e)\right)^2$. We start by defining the minimum/maximum of exogenous model parameters: $\bar{\lambda} := \max_{t \in [0,T], k \in [K], e \in E}{\lambda_t^{k,e}}$; $\ubar{\gamma}:=\min_{k\in[K]} \gamma^k$, $\bar{\gamma}:=\max_{k\in[K]} \gamma^k$; $\ubar{c}:=\min_{k \in [K]} \{c_{\nu}^{k}, c_{\sI}^{k}, c_{\lambda}^{k}\}$, $\bar{c}:=\max_{k \in [K]} \{c_{\nu}^{k}, c_{\sI}^{k}, c_{\lambda}^{k}\}$; $\bar{\beta} := \max_{k \in [K]}\beta^{k}$ ; $\bar{\eta} := \max_{k \in [K]} \eta^{k}$; $\bar{\kappa} := \max_{k \in [K]}\kappa^{k}$.  Denote $u^i(e) = (u^{i,k}_{t}(e))_{t\in[0,T],k\in[K]}$ and the corresponding equilibrium control $\hat{\alpha}^{i}(e) = (\hat{\alpha}^{i,k}_{t}(e))_{t\in[0,T],k\in[K]}$ for $e\in\{\sS, \sI, \sR\}$ and $\hat{\nu}^{i} = (\hat{\nu}^{i,k}_{t})_{t\in[0,T],k\in[K]}$, $i \in \{1,2\}$.
    We stress that the admissible set of socialization level is $[0,1]$, and further assume that the admissible set of vaccination level is $[0,V]$ with an upper bound $V > 0$.
    For $t \in [0,T]$, $k \in [K]$ and $e \in E$, we denote $\Delta p_t^{k}(e) = p^{1,k}_{t}(e) - p^{2,k}_{t}(e)$, $\Delta u_t^{k}(e) = u^{1,k}_{t}(e) - u^{2,k}_{t}(e)$, $\Delta \hat{\alpha}^{k}_t(e) = \hat{\alpha}^{1,k}(e) - \hat{\alpha}^{2,k}(e)$, $\Delta \hat{\nu}_{t}^{k} = \hat{\nu}_{t}^{1,k}-\hat{\nu}_{t}^{2,k}$ and $\Delta Z_{t}^{k} = Z_{t}^{1,k} - Z_{t}^{2,k}$.
    
    \par{\textbf{Step 0 (Boundedness and notational tools). }}We derive some bounds that will be frequently used in context.
    \begin{itemize}
        \item For aggregate function, consider $Z_{t}^{1,k}$ and $Z_{t}^{2,k}$, we have
        $$\left|\Delta Z_{t}^{k}\right| = \left|Z_{t}^{1,k} - Z_{t}^{2,k}\right| = \big| \sum_{l \in[K]} w(k, l) \lambda_t^{l, \sI} m^l \Delta p_t^l(\sI) \big|\leq\bar{W} \sum_{l \in [K]}|\Delta p_t^{l}(\sI)|,$$ where $\bar{W}:= \bar{\lambda}\argmax_{k,l} w(k, l) m^l$. Also, we have $\left|Z_t^k\right| \leq \bar{\lambda} \sum_{l \in[K]} w(k, l) m^l:=\bar{Z}$ by using the fact that state density flow is bounded by 1.
        \item For value function, we have its bound for all $t \in [0,T], k \in [K], e \in E$, by using the cost functional form $$\left|u_t^k(e)\right| \leq T\left[\bar{c}(\max(\bar{\lambda},A))^2+\bar{c} \max(1,V^2)\right] =: \bar{U}.$$
        \item We assume $|\hat{\alpha}_{t}^{k}(e)|\leq 1$ for simplicity for all $t \in [0,T], k \in [K], e \in E$. Consider $\hat{\alpha}_t^{1,k}$ and $\hat{\alpha}_t^{2,k}$, by adding and subtracting the cross-terms, we have
        $$\begin{aligned}
\left|\Delta \hat{\alpha}_t^k(\sS)\right| \leq & \frac{\beta^k}{2 c_\lambda^k}\left(\left|Z_t^{1, k}\right|\left|\Delta u_t^k(\sS)\right|+\left|u_t^{2, k}(\sS)\right|\left|\Delta Z_t^k\right|+\left|Z_t^{1, k}\right|\left|\Delta u_t^k(\sI)\right|+\left|u_t^{1, k}(\sI)\right|\left|\Delta Z_t^k\right|\right) \\
\leq & \frac{\bar{\beta}}{\ubarc} \bar{U}\left|\Delta Z_t^k\right|+\frac{\bar{\beta}}{2 \ubarc} \bar{Z}\left|\Delta u_t^k(\sS)\right|+\frac{\bar{\beta}}{2 \ubarc} \bar{Z}\left|\Delta u_t^k(\sI)\right| \\
\leq & \frac{\bar{\beta}}{\ubarc} \bar{U} \bar{W} \sum_{l \in[K]}\left|\Delta p_t^l(\sI)\right|+\frac{\bar{\beta}}{2 \ubarc} \bar{Z}\left|\Delta u_t^k(\sS)\right|+\frac{\bar{\beta}}{2 \ubarc} \bar{Z}\left|\Delta u_t^k(\sI)\right|.
\end{aligned}$$
        \item We assume $|\hat{\nu}_{t}^{k}| \leq V$ for simplicity for all $t \in [0,T], k \in [K]$. Consider $\hat{\nu}_t^{1,k}$ and $\hat{\nu}_t^{2,k}$, similarly
        $$\left|\Delta \hat{\nu}_t^k\right| \leq \frac{\bar{\kappa}}{2 \ubarc}\left|\Delta u_t^k(\sS)\right|+\frac{\bar{\kappa}}{2 \ubarc}\left|\Delta u_t^k(\sR)\right|.$$
    \end{itemize}

    \par{\textbf{Step 1 (Contraction boundedness of $h_1$). }}
    We start with two processes $p^{1} = (p^{1,k}_{t}(e))_{t\in[0,T],k\in[K],e\in E}$ and $p^{2} = (p^{2,k}_{t}(e))_{t\in[0,T],l\in[K],e\in E}$, and consider $\Delta p = p^1 - p^2$. First for state $\sS$, by chain rule and plugging in the dynamics of value function and finally by applying Young's inequality, we have 
    $$\begin{aligned}
        &\hskip-2mm\frac{d}{d t}\left\|\Delta u_t(\sS)\right\|^2\\
        &=2 \Delta u_t(\sS) \cdot \Delta \dot{u}_t(\sS) \\
        &=\sum_{k\in[K]} 2\Delta u^{k}_t(\sS)\Big[\beta^k\Big\{\Delta \hat{\alpha}_t^k(\sS) Z_t^{1, k}\left(u_t^{1, k}(\sS)-u_t^{1, k}(\sI)\right)+\hat{\alpha}_t^{2,k}(\sS)\Delta Z_t^{k}\left(u_t^{1, k}(\sS)-u_t^{1, k}(\sI)\right) \\ & \qquad\qquad\qquad\qquad  +\hat{\alpha}_t^{2,k}Z_t^{2,k}\left(\Delta u_t^{k}(\sS) - \Delta u_t^{k}(\sI)\right)\Big\} - c_\lambda^k \Delta \hat{\alpha}_t^k\left(\hat{\alpha}_t^{1, k}+\hat{\alpha}_t^{2, k}-2 \lambda_t^{k, \sS}\right) -c_{\nu}^{k}(\hat{\nu}^{1,k}_{t} + \hat{\nu}^{2,k}_{t}) \Delta \hat{\nu}^{k}_{t}\Big]\\
        &\leq \sum_{k\in[K]} 2 \Delta u^k_t(\sS)\{ [2 \bar{\beta} \bar{Z} \bar{U}+\bar{c}(2+2 \bar{\lambda})]\left|\Delta \hat{\alpha}_t^k(\sS)\right|+2 \bar{c} V\left|\Delta \hat{\nu}_t^k\right| +\bar{\beta} \bar{Z}\left|\Delta u_t^k(\sS)\right|+\bar{\beta} \bar{Z}\left|\Delta u_t^k(\sI)\right|+ 2\bar{U}\left|\Delta Z_t^{k}\right|\} \\
        &\leq \sum_{k\in[K]} 2\left|\Delta u^k_t(\sS)\right|\Bigg\{\Big([2 \bar{\beta} \bar{Z} \bar{U}+\bar{c}(2 +2 \bar{\lambda})] \bar{U} \bar{W} \frac{\bar{\beta}}{\ubar{c}} + 2\bar{U}\bar{W}\Big) \sum_{l \in[K]}\left|\Delta p_t^l(\sI)\right| + \frac{\bar{\kappa} \bar{c} V}{\ubar{c}}\left|\Delta u_t^k(\sR)\right| \\ & \quad\quad +\left(\frac{\bar{\kappa} \bar{c} V}{\ubar{c}}+[2 \bar{\beta} \bar{Z} \bar{U}+\bar{c}(2+2 \bar{\lambda})]\frac{\bar{\beta}}{2\ubar{c}}\bar{Z}+\bar{\beta}\bar{Z}\right) \left|\Delta u_t^k(\sS)\right| + \left([2 \bar{\beta} \bar{Z} \bar{U}+\bar{c}(2+2 \bar{\lambda})]\frac{\bar{\beta}}{2\ubar{c}}\bar{Z}+\bar{\beta}\bar{Z}\right) \left|\Delta u_t^k(\sI)\right| \Bigg\} \\ & 
        \leq d_1^{u}\sum_{k\in[K]}\left|\Delta u_t^{k}(\sS)\right|^2 + d_2^{u}\sum_{k\in[K]}\left|\Delta p_t^{k}(\sI)\right|^2 + d_3^{u}\sum_{k\in[K]}\left|\Delta u_t^{k}(\sI)\right|^2 + d_4^{u}\sum_{k\in[K]}\left|\Delta u_t^{k}(\sR)\right|^2, 
\end{aligned}
    $$
    where $d_1^u := \frac{3\bar{\kappa} \bar{c} V}{\ubar{c}}+[2 \bar{\beta} \bar{Z} \bar{U}+\bar{c}(2 +2 \bar{\lambda})] \frac{3\bar{\beta}}{2\ubar{c}} \bar{Z}+3 \bar{\beta}  \bar{Z} + [2 \bar{\beta} \bar{Z} \bar{U}+\bar{c}(2+2 \bar{\lambda})]\frac{\bar{\beta}}{\ubar{c}}\bar{U}\bar{W} + 2\bar{U}\bar{W}$, $d_2^{u}:=\big([2 \bar{\beta} \bar{Z} \bar{U}+\bar{c}(2+2 \bar{\lambda})] \bar{U} \bar{W} \frac{\bar{\beta}}{\ubar{c}} + 2\bar{U}\bar{W}\big)K$, $d_3^u := [2 \bar{\beta} \bar{Z} \bar{U}+c(2+2 \bar{\lambda})] \frac{\bar{\beta}}{2\ubar{c}}\bar{Z}+\bar{\beta}\bar{Z}$ and $d_4^{u}:=\frac{\bar{\kappa} \bar{c} V}{\ubar{c}}$. Similarly for state $\sI$ and state $\sR$, it follows
    $$\begin{aligned}
& \frac{d}{d t}\left\|\Delta u_t(\sI)\right\|^2 \leq 3 \bar{\gamma} \sum_{k\in[K]}\left|\Delta u_t^t(\sI)\right|^2+\bar{\gamma} \sum_{k\in[K]}\left|u_t^k(\sR)\right|^2, \\
& \frac{d}{d t}\left\|\Delta u_t(\sR)\right\|^2\leq3 \bar{\eta} \sum_{k\in[K]}\left|\Delta u_t^k(\sR)\right|^2+\bar{\eta} \sum_{k\in[K]}\left|\Delta u_t^k(\sS)\right|^2.
\end{aligned}
    $$
    We combine the three states to have
    $$\begin{aligned}
\frac{d}{d t}\left\|\Delta u_t\right\|^2 &= \frac{d}{d t}\left(\left\|\Delta u_t(\sS)\right\|^2+\left\|\Delta u_t(\sI)\right\|^2+\left\|\Delta u_t(\sR)\right\|^2\right) \\
& \leq \left(d_1^u+\bar{\eta}\right) \sum_{k\in[K]}\left(\Delta u_t^k(\sS)\right)^2+\left(d_3^u+3 \bar{\gamma}\right) \sum_{k\in[K]}\left(\Delta u_t^k(\sI)\right)^2 +\left(d_4^u+\bar{\gamma}+3 \bar{\eta}\right) \sum_{k\in[K]}\left(\Delta u_t^k(\sR)\right)^2\\ & \quad+d_2^u \sum_{k\in[K]}\left(\Delta p_t^k(\sI)\right)^2 \\ &\leq A_1 \| \Delta u_t\|^2 + A_2 \|\Delta p_t \|^2,
\end{aligned}$$
where $A_1 := \max\{d_1^u+\bar{\eta}, d_3^u+3 \bar{\gamma}, d_4^u+\bar{\gamma}+3 \bar{\eta}\}$ and $A_2 := d_2^{u}$. Then apply Gr\"{o}nwall's inequality to have
    $$\left\|\Delta u_t\right\|^2 \leq A_2 \int_t^T e^{A_1(t-s)}\left\|\Delta p_s\right\|^2 d s \leq A_2 e^{A_{1}T} \int_0^{T}\|\Delta p_s\|^2ds.$$

    \par{\textbf{Step 2 (Contraction boundedness of $h_2$). }} For the two processes $\tilde{p}^1 := h(p^1)$ and $\tilde{p}^2 := h(p^2)$, denote $\Delta \tilde{p} = \tilde{p}^1 - \tilde{p}^2$. First consider state $\sS$, by chain rule and plugging in the dynamics of state density flow and finally by applying Young's inequality, we have
    $$\begin{aligned}
        \frac{d}{dt} \left\|\Delta \tilde{p}_t(\sS)\right\|^2 &= 2\Delta \tilde{p}_t(\sS)\cdot\Delta \dot{\tilde{p}}_t(\sS) \\ &= 
\sum_{k\in[K]} 2 \Delta \tilde{p}_t^k(\sS) \Big[\sum_{l \in[K]} w(k, l) \lambda_t^{l, \sI} m^l\left(-\beta^k \hat{\alpha}_t^{1,k}(\sS) \tilde{p}_t^{1,l}(\sI) \tilde{p}_t^{1,k}(\sS)+\beta^{k} \hat{\alpha}_t^{2, k}(\sS) \tilde{p}_t^{2,l}(\sI) \tilde{p}_t^{2, k}(\sS)\right) \\ &\quad \quad -\kappa^k\left(\hat{\nu}_t^{1,k} \Delta \tilde{p}_t^k(\sS)+\tilde{p}_t^{2, k}(\sS) \Delta \hat{\nu}_t^k\right)+\eta^k \Delta \tilde{p}_t^k(\sR)\Big] \\ &\leq 
\sum_{k\in[K]} 2\left|\Delta \tilde{p}_t^k(\sS)\right|\Big[\bar{W} \bar{\beta} \sum_{l\in[K]}\left(\left|\Delta \tilde{p}_t^k(\sS)\right|+\left|\Delta \tilde{p}_t^l(\sI)\right|+|\Delta \hat{\alpha}_t^k(\sS)|\right)+\bar{\kappa} V\left|\Delta \tilde{p}_t^k(\sS)\right| +\bar{\kappa}\left|\Delta \hat{\nu}_t^k\right|+\bar{\eta}\left|\Delta \tilde{p}_t^k(\sR)\right|\Big]
\\ & \leq \sum_{k\in[K]} \Big(3K \bar{W} \bar{\beta} +2 \bar{\kappa} V+K^2 \bar{W}^2\bar{U}\frac{\bar{\beta}^2}{\ubar{c}}+K \bar{W} \bar{Z}\frac{\bar{\beta}^2}{\ubar{c}}+\frac{\bar{\kappa}^2}{\ubar{c}}+\bar{\eta}\Big)\left|\Delta \tilde{p}_t^k(\sS)\right|^2 + \Big(K\bar{W}\bar{\beta} + K^2\bar{W}^2\bar{U}\frac{\bar{\beta}^2}{\ubar{c}}\Big)\left|\Delta \tilde{p}_t^k(\sI)\right|^2 \\ & \quad\quad + \bar{\eta} \left|\Delta \tilde{p}_t^k(\sR)\right|^2 + \Big(K \bar{W} \bar{Z} \frac{\bar{\beta}^2}{2 \ubar{c}}+\frac{\bar{\kappa}^2}{2 \ubar{c}}\Big)|\Delta u_t^{k}(\sS)|^2 + K \bar{W} \bar{Z} \frac{\bar{\beta}^2}{2\ubar{c}}\left|\Delta u_t^k(\sI)\right|^2+\frac{\bar{\kappa}^2}{2 \ubar{c}}\left|\Delta \tilde{u}_t^k(\sR)\right|^2.
    \end{aligned}$$
For state $\sI$ and $\sR$, we similarly have
$$\begin{aligned}
    \frac{d}{dt} \left\|\Delta \tilde{p}_t(\sI)\right\|^2 &= 2\Delta \tilde{p}_t(\sI)\cdot\Delta \dot{\tilde{p}}_t(\sI) \\ &\leq \sum_{k\in[K]} \Big(3K\bar{W} \bar{\beta}+2K^2\bar{W}^2 \bar{\beta}^2 \frac{\bar{U}}{\ubar{c}} +K\bar{W} \bar{\beta}^2 \frac{\bar{Z}}{\ubar{c}}+2 \bar{\gamma}\Big)\left|\Delta \tilde{p}_t^k(\sI)\right|^2 + K\bar{W} \bar{\beta} \left|\Delta \tilde{p}_t^k(\sS)\right|^2 \\ &\quad\quad + K\bar{W} \bar{\beta}^2 \frac{\bar{Z}}{2\ubar{c}} \left|\Delta u_t^k(\sS)\right|^2 + K\bar{W} \bar{\beta}^2 \frac{\bar{Z}}{2\ubar{c}}\left|\Delta u_t^k(\sI)\right|^2,  \\ \frac{d}{dt} \left\|\Delta \tilde{p}_t(\sR)\right\|^2 &= 2\Delta \tilde{p}_t(\sR)\cdot\Delta \dot{\tilde{p}}_t(\sR) \\ &\leq \sum_{k\in[K]}\Big(\bar{\gamma}+\frac{\bar{\kappa}^2}{\ubar{c}}+\bar{\kappa} V+2 \bar{\eta}\Big)\left|\Delta \tilde{p}_t^k(\sR)\right|^2+\frac{\bar{\kappa}^2}{2 \ubar{c}}\left|\Delta u_t^k(\sS)\right|^2+\frac{\bar{\kappa}^2}{2 \ubar{c}}\left|\Delta u_t^k(\sR)\right|^2 + \bar{\kappa}V \left|\Delta \tilde{p}_t^k(\sS)\right|^2 + \bar{\gamma} \left|\Delta \tilde{p}_t^k(\sI)\right|^2.
\end{aligned}$$
We combine the three inequalities to have
$$\begin{aligned}
    \frac{d}{d t}\left\|\Delta \tilde{p}_t\right\|^2&=\frac{d}{d t}\left(\left\|\Delta \tilde{p}_t(\sS)\right\|^2+\left\|\Delta \tilde{p}_t(\sI)\right\|^2+\left\|\Delta \tilde{p}_t(\sR)\right\|^2\right)\\ &\leq d_1^p \sum_{k\in[K]}\left|\Delta \tilde{p}_t^k(\sS)\right|^2+d_2^p \sum_{k\in[K]}\left|\Delta \tilde{p}_t^k(\sI)\right|^2+d_3^p \sum_{k\in[K]}\left|\Delta \tilde{p}_t^k(\sR)\right|^2+d_4^p \sum_{k\in[K]}\left|\Delta u_t^k(\sS)\right|^2 \\ & \quad\quad +d_5^p \sum_{k\in[K]}\left|\Delta u_t^k(\sI)\right|^2 + d_6^p \sum_{k\in[K]}\left|\Delta u_t^k(\sR)\right|^2  \\ & \leq B_1 \|\Delta \tilde{p}_t\|^2 + B_2 \|\Delta u_t\|^2,
\end{aligned}$$
where $d_1^p := 4 K\bar{W} \bar{\beta}+3 \bar{\kappa} V+\bar{\eta}+\frac{\bar{\kappa}^2}{c}+K\bar{W} \bar{\beta}^2 \frac{\bar{Z}}{\ubar{c}}+K^2\bar{W}^2 \bar{\beta}^2 \frac{\bar{U}}{\ubar{c}}$, $d_2^p := 4 K\bar{W} \bar{\beta}+3 \bar{\gamma}+K\bar{W} \bar{\beta}^2 \frac{\bar{Z}}{\ubar{c}}+3K^2 \bar{W}^2 \bar{\beta}^2 \frac{\bar{U}}{\ubar{c}}$, $d_3^p:=3\bar{\eta}+\bar{\gamma}+\bar{\kappa}V+\frac{\bar{\kappa}^2}{\ubar{c}}$, $d_4^p:=K\bar{W} \bar{\beta}^2 \frac{\bar{Z}}{\ubar{c}} +\frac{\bar{\kappa}^2}{\ubar{c}}$, $d_5^p:=K\bar{W} \bar{\beta}^2 \frac{\bar{Z}}{\ubar{c}}$, $d_6^p:=\frac{\bar{\kappa}^2}{\ubar{c}}$; $B_1:=\max \left\{d_1^p, d_2^p, d_3^p\right\}$ and $B_2:=\max \left\{d_4^p, d_5^p, d_6^p\right\}$. Apply Gr\"{o}nwall's inequality to have $$\left\|\Delta \tilde{p}_t\right\|^2 \leq B_2 \exp \left(B_1 T\right) \int_0^T\left\|\Delta u_s\right\|^2 d s$$ Plug in the result from step 1 to achieve
$$\left\|\Delta \tilde{p}_t\right\|^2 \leq B_2 \exp \left(B_1 T\right) \int_0^T A_2 \exp(A_{1}T) \int_0^T\left\|\Delta p_s\right\|^2 d s d v$$
Take supremum over $t$ on both sides and simplify to conclude
$$\|\Delta \tilde{p}\|_T = \sup_{t\in [0,T]}\|\Delta\tilde{p}_t\|^2 \leq A_2 B_2 \exp(A_{1}T) \exp \left(B_1 T\right) T^2 \sup _t\|\Delta p_t\|^2 =: C_1 \|\Delta p\|_T$$
Under small time horizon $T$, there is $C_1 = A_2 B_2 \exp \left((A_1 + B_1) T\right) T^2 < 1$ and $h$ is a contraction mapping. By Banach fixed point theorem, we conclude the unique existence of the FBODE solution. This completes the proof. We remark that if one only aims to establish the existence of a Nash equilibrium, a milder small-time assumption ensuring function continuity would suffice, and the result could then be obtained by the Schauder fixed-point theorem.
\end{proof}

\section{Numerical Algorithm}
\label{subsec:alg}
Following the theoretical results presented in Section~\ref{sec:sir_theoretical_result}, in order to find a multi-population MFG Nash equilibrium for our model, we will need to solve the FBODE system given in Theorem~\ref{the:fbode_sir}. We will solve these coupled forward and backward equations by implementing a fixed point algorithm as explained in the main text, see Algorithm~\ref{algo:vacc}. In particular, the state distribution dynamics follow the forward component of the ODE system and they are initialized with the given initial state densities. On the other hand, the value function dynamics follow the backward component of the ODE system and they are set to zeros for all states since there is no terminal cost. Since aggregate depends only on state distribution and known model parameters, we could decouple it from equilibrium socialization levels and vaccination rates, and compute each in order. Then we iteratively update each variable until convergence.
{
\begin{algorithm}
\caption{ Multi-population MFG Nash Equilibrium \label{algo:vacc}}

\textbf{Input:} Time horizon: $T$; Time increments: $\Delta t$; Initial state distributions in group $k\in[K]$: $p^k(\cdot)$; Model parameters: $\beta^k,\kappa^k,\gamma^k, c_{\lambda}^k, c_{\nu}^k, c_{\sI}^k, \blambda = (\lambda_t^{k,\sS}, \lambda_t^{k,\sI}, \lambda_t^{k,\sR})_{t\in\{0, \Delta t, 2\Delta t, \dots, T\}}$ for $k \in [K]$; Group connection matrix: $(w(k,l))_{k,l \in [K]}$;  Convergence parameter: $\epsilon$.\vskip1mm

\textbf{Output:} Nash equilibrium contact levels and vaccination decisions for representative individuals in each group $k\in[K]$: $\hat{\balpha}^k$, $\hat{\bnu}^k$; State distributions at Nash equilibrium in each group $k\in[K]$: $\bp^k$.

\vskip2mm

\begin{algorithmic}[1]
\STATE Initialize ${\bp}^{k,(0)}(e)=({p}_0^{k,(0)}(e),p_{\Delta t}^{k,(0)}(e),\dots, p_{T}^{k,(0)}(e))$ and ${\bu}^{k,(0)}(e)=({u}_0^{k,(0)}(e),u_{\Delta t}^{k,(0)}(e),\dots, u_{T}^{k,(0)}(e))$ \\ for $k\in[K], e\in\{\sS, \sI, \sR\}$.\vskip2mm
\WHILE{$\lVert \bp^{k,(j)}- \bp^{k,(j-1)}\rVert>\epsilon$ \OR $\lVert \bu^{k,(j)}- \bu^{k,(j-1)}\rVert>\epsilon$, 
 for any $k\in[K]$}\vskip2mm
    \STATE{Calculate $Z_t^{k, (j)}$ for all $k\in[K]$ by using the last equation in FBODE~\eqref{eq:FBODE} and equation~\eqref{eq:alphaI}.}\vskip1.5mm
    \STATE{Calculate best response controls: $(\alpha_t^{k,(j)}(\sS), \alpha_t^{k,(j)}(\sI), \alpha_t^{k,(j)}(\sR), \nu_t^{k,(j)})$ by plugging in $\bu^{k,(j)}$, $\bZ^{k,(j)}$ in equations~\eqref{eq:alphaS}-\eqref{eq:Nu} for all $k\in[K]$.}\vskip1.5mm
    \STATE{Update $\bp^{k, (j+1)}$ by solving the forward ODEs in~\eqref{eq:FBODE} by using $\balpha^{(j)}, \bZ^{(j)}, \bu^{(j)}$ for all $k\in[K]$.}\vskip1.5mm
    \STATE{Update $\bu^{k, (j+1)}$ by solving the backward ODEs in~\eqref{eq:FBODE} by using $\balpha^{(j)}, \bZ^{(j)}, \bp^{(j)}$ for all $k\in[K]$.}\vskip1.5mm
\ENDWHILE
\vskip2mm
\STATE{Calculate $\hat{\bZ}^{k}$ for all $k\in[K]$ by using the last equation in FBODE~\eqref{eq:FBODE} and equation~\eqref{eq:alphaI}.}\vskip2mm
\STATE{Calculate best response controls: $\hat{\balpha}^k, \hat{\bnu}^k$ by plugging in $\bu^{k,(j+1)}$, $\hat{\bZ}$ in equations~\eqref{eq:alphaS}-\eqref{eq:Nu} for all $k\in[K]$.}\vskip2mm
\RETURN $(\hat{\balpha}^k, \hat{\bnu}^k, \bp^{k,(j+1)})_{k\in[K]}$.
\end{algorithmic}
\end{algorithm}}

\section{Population Structure and Model Parameterization}

\noindent\textbf{Population Structure:} Stratifying the population by income and perceived authority revealed consistent trends across the survey questions that guided the model construction and parametrization. Details on how we defined the different population groups are stated below.\\
\vspace{-6pt}

\noindent\textit{Income classification}: We used survey outcomes to classify individuals into different economic statuses, using the question `Which of the following best describes your yearly combined household income from all sources?'. Responses from this question were then assigned into \textit{low, middle,} and \textit{high} income categories according to the following brackets respectively: $\$49,999$ or less, $\$50,000-\$99,999$, and $\$100,000$ or more.\\
\vspace{-6pt}

\noindent\textit{Perception of authority}: Using the question `Which of the following best describes your political ideology?' we classified individuals as \textit{follower} if their answer was `Liberal', and \textit{indifferent} if they answered `Conservative', based on the distrust of the government (Table~\ref{table:costVac}). We filtered out the following responses due to small sample sizes and heterogeneity that prevented alignment with our binary classification: `left-wing, but I do not consider myself a liberal', `Right-wing, but I do not consider myself a conservative' and `Other'. Following these exclusions, 8,991 respondents remained for analysis. Based on the income classification and perception of authority, the numbers and proportions of the six distinct groups are show in Table \ref{table:groups}.

\begin{table}[h!]
    \footnotesize
    \centering
    \begin{tabular}{l|l}
    \toprule
        Groups & Respondents (percentage) \\ 
        \midrule \midrule 
        Low SES -- Follower (LF) & 802 (14.7 $\%$) \\ \hline
        Low SES -- Indifferent (LI) & 953 (17.5 $\%$) \\ \hline
        Mid SES -- Follower (MF) & 822 (15.1 $\%$)  \\ \hline
        Mid SES -- Indifferent (MI) & 901 (16.6 $\%$)\\ \hline
        High SES -- Follower (HF) & 1012 (18.6 $\%$) \\ \hline
        High SES -- Indifferent (HI) & 954 (17.5 $\%$)\\ \hline
    \bottomrule 
    \end{tabular}
    \caption{Distribution of respondents by socioeconomic group (SES) and perception of authority}
    \vspace{-0.1cm} 
    \label{table:groups}
\end{table}

\noindent\textbf{Model parametrization}: Parameters were determined through a combination of survey indicators, literature review, and assumptions (Table \ref{table:params}). The survey results were either used as observed values or scaled to illustrate differences between groups, rather than as absolute measurements to account for factors not captured by the study.\\

\begin{table}[h!]
    \footnotesize
    \centering
    \begin{tabular}{l|l|l|l|l|l|l|l|l}
    \toprule
        Parameter & Description & LF & LI & MF & MI & HF & HI & Reference \\ \midrule \midrule 
        $\beta$ &  Base transmission rate & 0.4 & 0.4 & 0.35 & 0.35 & 0.3 & 0.3 & \cite{Liu2020, Mena2021} \\ \hline
        $\gamma$ & Recovery rate (1/days) & 0.143 & 0.143 & 0.143 & 0.143 & 0.143 & 0.143 & \cite{VanKampen2021} \\ \hline
        $\eta$ & Waning of immunity rate (1/days) & 0.004 & 0.004 & 0.004 & 0.004 & 0.004 & 0.004 & \cite{Abu-Raddad2021, Hall2021} \\ \hline        
        $I_0$ & Initial proportion of infected individuals & 1\% & 1\% & 1\% & 1\% & 1\% & 1\% & -- \\ \hline
        $\kappa$ & Vaccine effectiveness rate & 0.03 & 0.03 & 0.03 & 0.03 & 0.03 & 0.03 & --  \\ \hline
        $c_{\sI}$ & Cost of treatment & 1.05 & 1.05 & 1 & 1 & 0.8 & 0.8 & Survey \\ \hline
        $c_\nu$ & Cost of vaccination & 1.4 & 1.6 & 1.2 & 1.4 & 0.8 & 1 & Survey\\ \hline
        $c_\lambda$ & \vtop{\hbox{\strut Importance given to following social distancing}\hbox{\strut guidelines or maximum socialization level} \hbox{\strut when susceptible}} & 1 & 1 & 1 & 1 & 1 & 1 & Survey\\ \hline
        $\xi_{\sI}$ & Sickness related intrinsic socialization & - & 0.97 & - & 0.97 & - & 0.97 & Survey\\ \hline
    \bottomrule 
    \end{tabular}
    \caption{Parameter symbol, description, and reference for each group listed Table \ref{table:groups}.}
    \vspace{-0.1cm} 
    \label{table:params}
\end{table}

\noindent\textit{Disease dynamics:} The values of $\beta$, $\gamma$, and $\eta$ are based on initial estimates of SARS-CoV-2 \cite{Liu2020, VanKampen2021, Abu-Raddad2021, Hall2021}. The infection rates ($\beta$) were adjusted to account for differences in exposure and disease risk across socioeconomic groups \cite{Mena2021}, while the recovery ($\gamma$) and waning of immunity ($\eta$) rates are taken to be independent of economic status and authority perception.\\
\vspace{-7pt}

\noindent\textit{Vaccine effectiveness rate $\kappa$:} This rate is assumed to be consistent across SES, and its value was selected to obtain infection outcomes that are consistent with those observed in the presence of vaccination.\\
\vspace{-7pt}

\noindent\textit{Cost of treatment $c_{\sI}$:} This parameter was derived from responses related to financial or access issues (lack of money, time, or transportation) to the question `Which of the following statements describes the main reasons you did not receive medical care?'. We assumed that lack of resources is independent of authority perception, resulting in 6 respondents in the low-income group, 5 in the mid-income group, and 1 in the high-income group for this question. While these numbers are low, they still suggest a trend at the socioeconomic level, and thus, we used these values to parametrize $c_{\sI}$ such that compared to the middle group, the low group has cost of treatment 5\% higher, and the high group 20\% lower, representing the observed differences. \\ \vspace{-7pt}

\noindent\textit{Cost of vaccination $c_\nu$:} This parameter measures perceived barriers or hesitations towards COVID-19 vaccination, based on responses to the question `Which of the following statements describe your main reasons why you have not been vaccinated and/or boosted?' (Table \ref{table:costVac}). When considering all the reasons, respondents identified as \textit{indifferent} consistently reported higher frequencies of these belief related reasons than \textit{followers} across all income levels, indicating different perceptions of the vaccination cost. Within each perception of authority group, high income respondents showed slightly lower levels of concern and distrust than mid or low income groups. In addition, resources barrier reasons such as lack of money, time, or transportation were more frequently reported among low- and mid-income respondents. Using these responses, we scaled $c_\nu$ values to capture both perceptional and income-related gradients: indifferent assigned higher values than follower and lower-income groups slightly higher than higher-income groups. \vspace{-7pt}
\\

\begin{table}[h!]
    \footnotesize
    \centering
    \begin{tabular}{l|l|l|l|l|l|l}
    \toprule
    Response & LF & LI & MF & MI & HF & HI \\ \hline
    \midrule
    Total & 298 & 501 & 218 & 415 & 131 & 403 \\ \hline
    Distrust of pharmaceutical companies & 31 (10.4\%) & 98 (19.6\%) & 26 (11.9\%) & 104 (25.1\%) & 13 (9.9\%) & 93 (23.1\%) \\ \hline
    Distrust of the government & 37 (12.4\%) & 102 (20.4\%) & 23 (10.6\%) & 106 (25.5\%) & 10 (7.6\%) & 87 (21.6\%) \\ \hline
    I am worried about adverse effects & 77 (25.8\%) & 133 (26.5\%) & 62 (28.4\%) & 138 (33.3\%) & 41 (31.3\%) & 145 (36.0\%) \\ \hline
    I don't believe in vaccines & 19 (6.4\%) & 58 (11.6\%) & 12 (5.5\%) & 29 (7.0\%) & 5 (3.8\%) & 29 (7.2\%) \\ \hline
    Lack of money & 23 (7.7\%) & 24 (4.8\%) & 15 (6.9\%) & 21 (5.1\%) & 4 (3.1\%) & 18 (4.5\%) \\ \hline
    Lack of time & 39 (13.1\%) & 25 (5.0\%) & 35 (16.1\%) & 23 (5.5\%) & 16 (12.2\%) & 23 (5.7\%) \\ \hline
    Lack of transportation & 22 (7.4\%) & 22 (4.4\%) & 7 (3.2\%) & 9 (2.2\%) & 4 (3.1\%) & 10 (2.5\%) \\
    \hline
    \bottomrule
    \end{tabular}
    \caption{Responses to the question `Which of the following statements describe your main reasons why you have not been vaccinated and/or boosted?'}
    \vspace{-0.1cm} 
    \label{table:costVac}
\end{table}

\noindent\textit{Importance given to following social distancing guidelines or socialization when susceptible $c_\lambda$:}  
This parameter represents the importance given by susceptible individuals to maintaining social interactions or adhering to preventive measures depending on their authority perception. It is informed by the question: `Have you changed your behavior (for example, mask-wearing or reducing contact) in response to your concern to COVID-19?' (Table~\ref{table:behav}). On average across income levels, followers were more likely to report behavioral change than indifferent individuals (52\% versus 35\%). Since this difference is already included in the model’s behavioral classification, to prevent double-counting this effect, we set $c_{\lambda} = 1$ uniformly across all  following the results of the survey; however, we emphasize our model allows group-dependent values. 

\begin{table}[h!]
    \footnotesize
    \centering
    \begin{tabular}{l|l|l|l|l|l|l}
    \toprule
    Response & LF & LI & MF & MI & HF & HI \\ \hline
    \midrule
    Yes & 406 (51.6\%) & 341 (36.6\%) & 434 (53.1\%) & 307 (34.4\%) & 528 (52.4\%) & 308 (32.6\%) \\ \hline
    \bottomrule
    \end{tabular}
    \caption{Responses to the question `Have you changed your behavior (for example, mask-wearing or reducing contact) in response to your concern to COVID-19?'}
    \vspace{-0.1cm} 
    \label{table:behav}
\end{table}

\noindent\textit{Sickness-related intrinsic socialization $\xi_{\sI}$:}  
This parameter represents the intrinsic tendency of individuals to reduce social contact when sick and it was informed by the question: `While you were sick with COVID-19, did you reduce your contact with other people?' As shown in Table~\ref{table:behsick}, the majority of respondents across all income and ideology groups reported reducing contact, with the proportion of `Yes' responses ranging from 78\% to 88\%. These results indicate that respondents regardless of income or perception of authority exhibited decreases in socialization when ill. Although the question does not specify the exact extent of their contact reduction, we set the intrinsic socialization level parameter to $\xi_I = 0.97$ for all indifferent groups. 

\begin{table}[h!]
    \footnotesize
    \centering
    \begin{tabular}{l|l|l|l|l|l|l}
    \toprule
    Response & LF & LI & MF & MI & HF & HI \\ \hline
    \midrule
    Yes & 265 (85.5\%) & 241 (78.0\%) & 322 (84.3\%) & 264 (79.0\%) & 441 (88.2\%) & 333 (81.2\%) \\ \hline
    \bottomrule
    \end{tabular}
    \caption{Responses to the question `While you were sick with COVID-19, did you reduce your contact with other people?'}
    \vspace{-0.1cm} 
    \label{table:behsick}
\end{table}

\noindent\textbf{{Contact levels:}} Table \ref{table:connection_matrix} defines the relative contact patterns among the groups. Within each income level, follower and indifferent individuals interact most frequently with others of the same perception of authority group, diagonal elements of the matrix set to 1 to represent full within-group connection. Within the same income level or within the same perception of authority, connections are set to be slightly lower (0.95), and interactions between groups that differ in both income and perception have even lower contact levels (0.9). This symmetric structure reflects the culture independent tendency for individuals to interact most frequently with others of similar socioeconomic status \cite{hilman2022socioeconomic,zwier2023segregation}, and within the ideology group, individuals show stronger social alignment and preference for interaction with members of their own group~\cite{balliet2018political}.

\begin{table}[h!]
    \footnotesize
    \centering
    \begin{tabular}{l|| l|l|l|l|l|l}
         & LF & LI & MF & MI & HF & HI \\ \hline
        \midrule
        LF & 1 & 0.95 & 0.95 & 0.9 & 0.95 & 0.9 \\ \hline
        LI & 0.95 & 1 & 0.9 & 0.95 & 0.9 & 0.95 \\ \hline
        MF & 0.95 & 0.9 & 1 & 0.95 & 0.95 & 0.9 \\ \hline
        MI & 0.9 & 0.95 & 0.95 & 1 & 0.9 & 0.95 \\ \hline
        HF & 0.95 & 0.9 & 0.95 & 0.9 & 1 & 0.95 \\ \hline
        HI & 0.9 & 0.95 & 0.9 & 0.95 & 0.95 & 1 \\ \hline
    \bottomrule
    \end{tabular}
    \caption{Contact matrix}
    \vspace{-0.1cm}
    \label{table:connection_matrix}
\end{table}

\section{Sensitivity analysis and additional simulation results for the SIR-type model}

In this section, we present more sensitivity analysis results on the SIR model parameters. In the experiments, by default we use permissive social distancing guidelines unless otherwise is specified. Figure~\ref{fig:supp_xi} compares two intrinsic social interaction levels of authority indifferents (i.e., sensitivity analysis on $\xi_{\sI}$). When authority indifferents maintain a higher intrinsically intrinsic socialization level, individuals in the society feel unsafe and would like to lower their socialization activities and elevate vaccination uptake. Despite this, we can see a higher infection curve for all groups. Furthermore, the individuals who are authority followers and lower income status have the most severe infection outbreak. Following this observation, for a population where infected and authority indifferent individuals have high intrinsic socialization level, stricter public health policies could be imposed to offset the infection impact they bring.

Figure~\ref{fig:supp_kappa} compares two levels of vaccination efficacies (i.e., sensitivity analysis on $\kappa$). Higher efficacy substantially motivates higher levels of vaccination uptake across all groups, which in turn encourages higher socialization levels. Overall, enhanced vaccine efficacy benefits the entire population by effectively mitigating infection spread. This shows that a public health authority can prioritize the scientific investments to increase in vaccination efficacy or focus on broadcasting the vaccine's efficacy to a wider audience as a mitigation policy.

Figure~\ref{fig:supp_c_inf} compares the outcomes under group-adaptive infection costs and uniformly low infection costs (i.e., sensitivity analysis on $c_{\sI}$. For the group-adaptive parameters, lower-income groups face higher infection costs, reflecting their large burden of treatment expense. Under uniformly low infection costs, lower income individuals perceive less anticipated price for getting infected, leading to their higher socialization levels and lower vaccination uptake. In contrast, the strategies of higher-income groups remain mostly unchanged, as vaccination continues to be the most cost-effective choice for them.

Figure~\ref{fig:supp_eta} compares two waning rates of immunity where a higher waning rate of immunity means a higher and earlier risk of reinfection (i.e., sensitivity analysis on $\eta$). As a result of this risk, individuals choose to delay vaccination in the early periods and instead lower their socialization levels over the time interval. Due to the higher risk of infection, the proportions of infected and also susceptible reach higher levels.
Because reinfections are more likely when $\eta$ is higher, vaccination rates remains moderately higher in the later periods for increased protection. Overall, it can be difficult to identify if a disease has high waning rate of immunity at early periods of spread, as infection curves appear similar and the fact that limited patient cases are available during the early periods of epidemics spread in real world. Stricter social-distancing policies or stronger vaccination incentives could therefore help prevent large-scale infections.

\begin{figure*}[t]
    \centering
    \subfloat[\textbf{Sensitivity analysis on $\xi_{\sI}$}]{
        \includegraphics[width=0.45\textwidth]{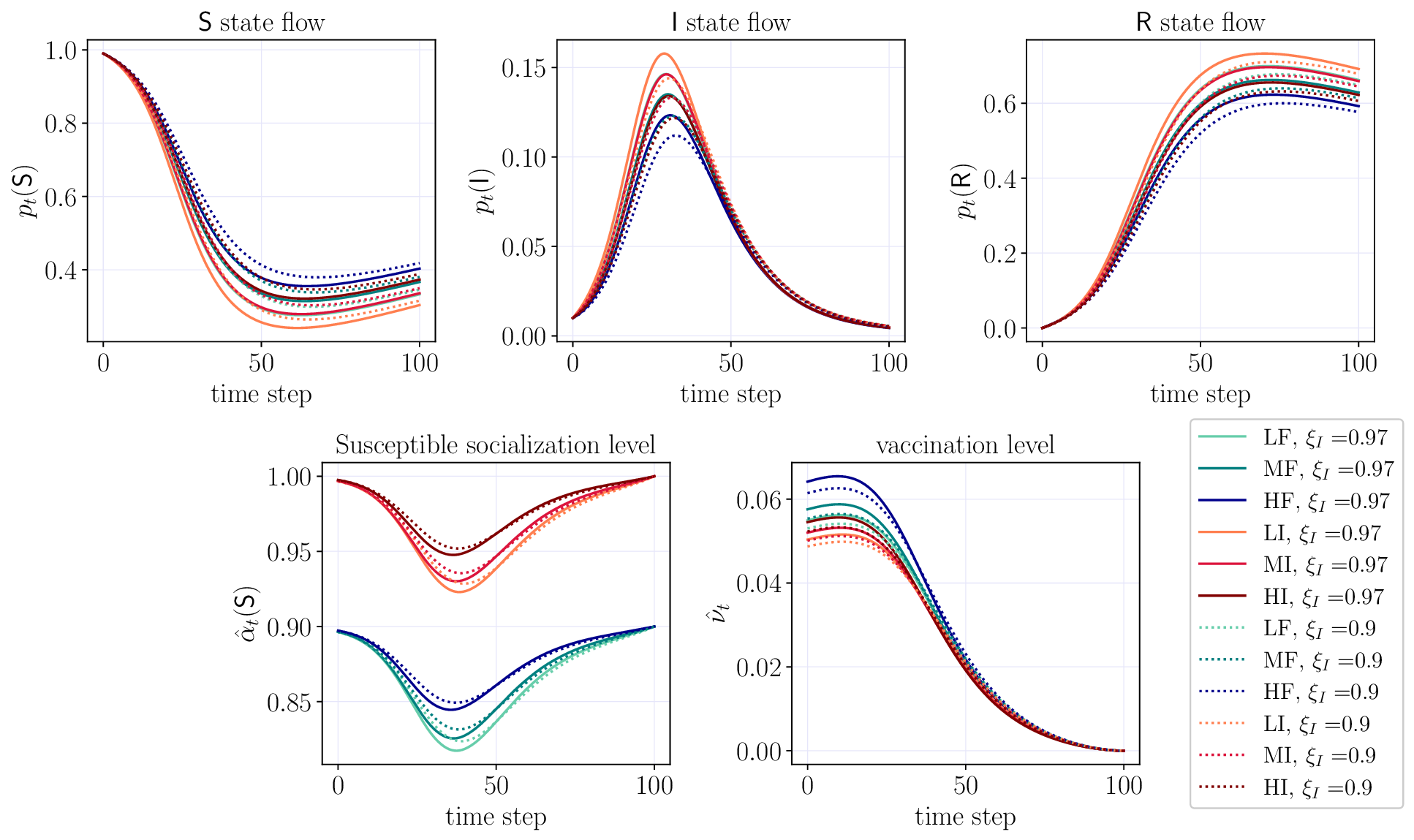}
        \label{fig:supp_xi}
    }
    \hfill
    \subfloat[\textbf{Sensitivity analysis on $\kappa$}]{
        \includegraphics[width=0.45\textwidth]{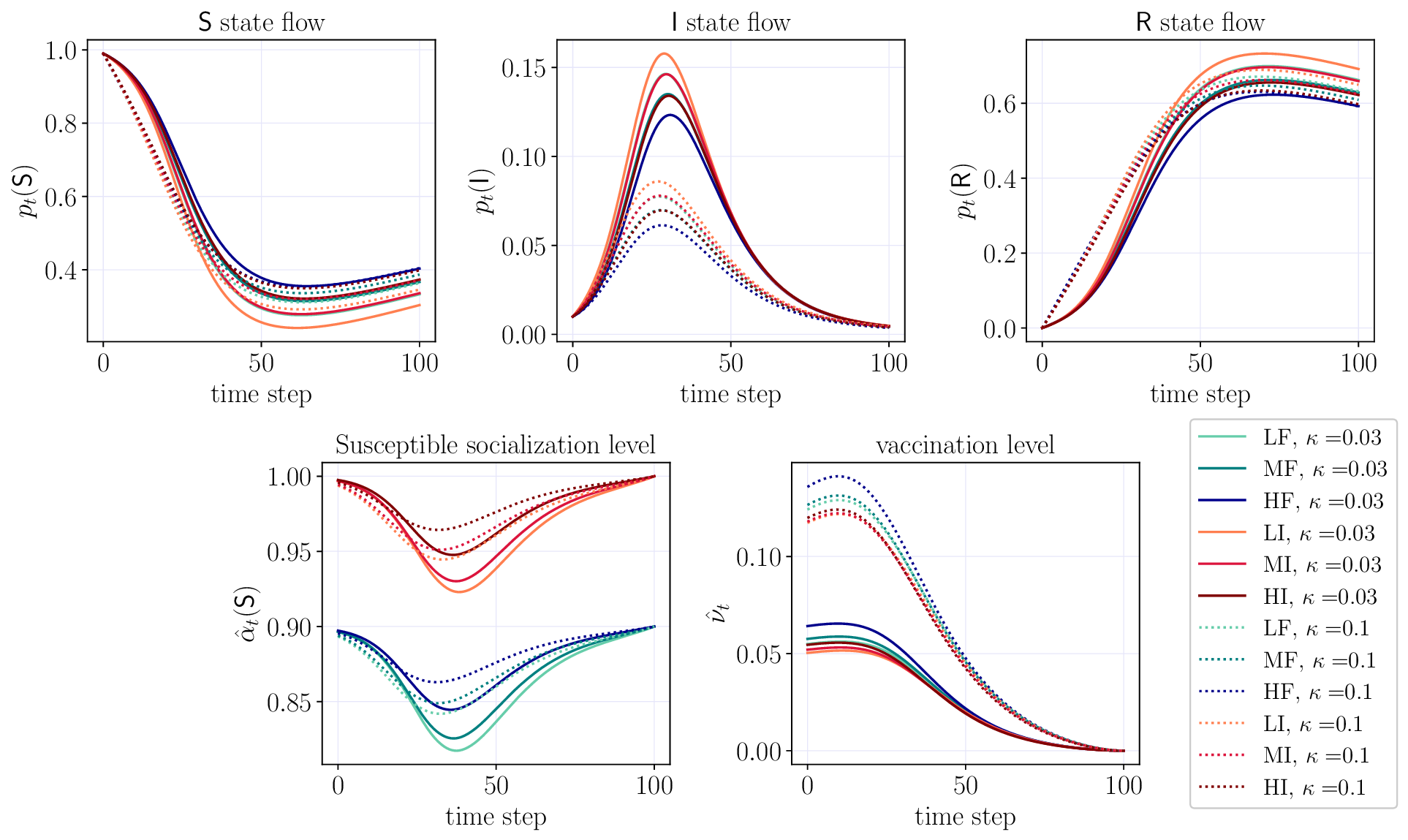}
        \label{fig:supp_kappa}
    }
    \captionsetup{font=footnotesize}
    \caption{
        Sensitivity analysis on SIR model parameters. Left: comparison of epidemics mitigation effects under two authority indifferents intrinsic socialization levels: $\xi_{\sI} = 0.9$; $\xi_{\sI} = 0.97$. Right: comparison of epidemics mitigation effects under two vaccination efficacy levels: $\kappa^k = 0.03$ for all groups; $\kappa^k = 0.1$ for all groups.
    }
    \label{fig:xi_kappa_comparison}
\end{figure*}

\begin{figure*}[t]
    \centering
    \subfloat[\textbf{Sensitivity analysis on $c_{\sI}$}]{
        \includegraphics[width=0.45\textwidth]{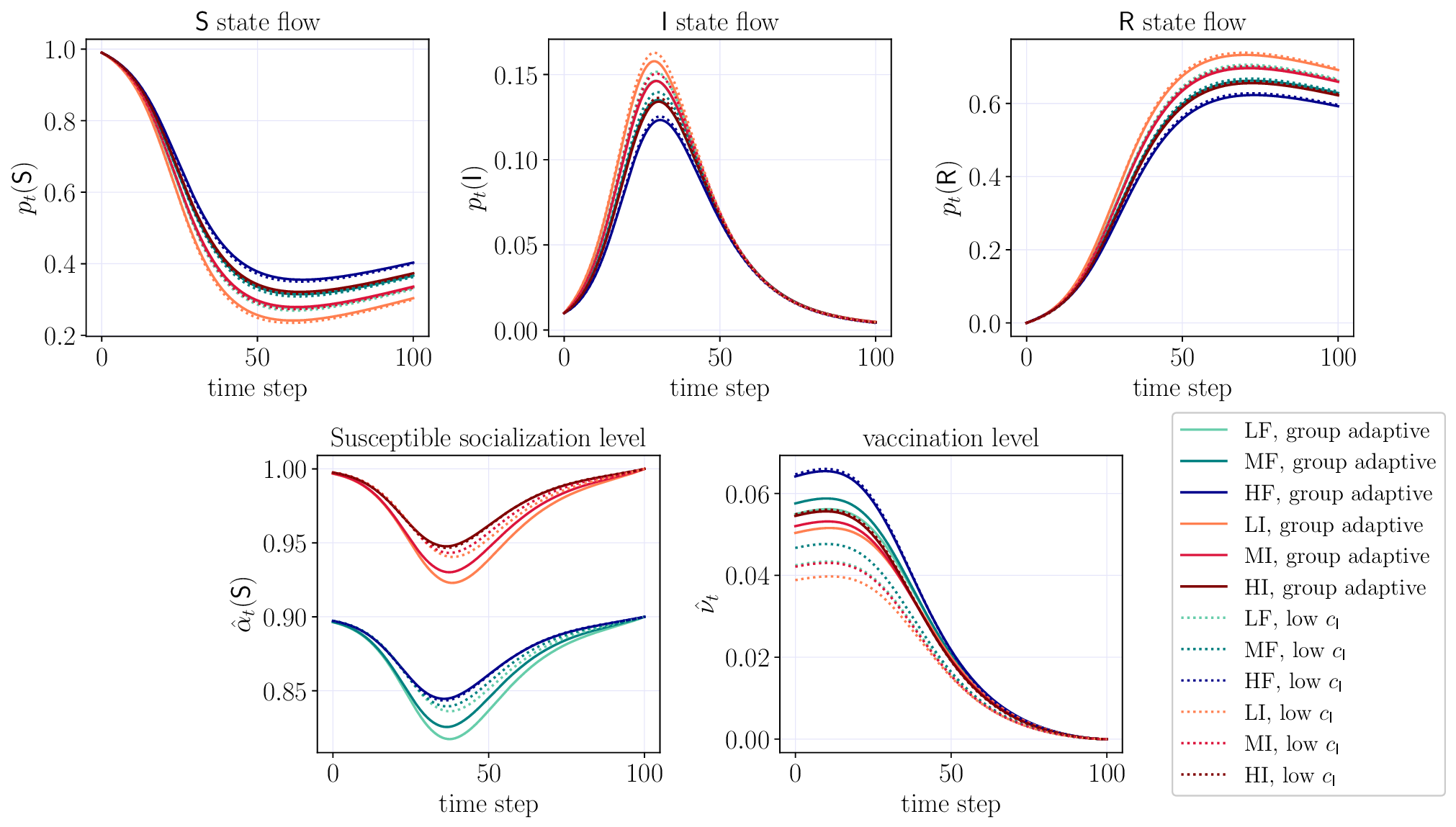}
        \label{fig:supp_c_inf}
    }
    \hfill
    \subfloat[\textbf{Sensitivity analysis on $\eta$}]{
        \includegraphics[width=0.45\textwidth]{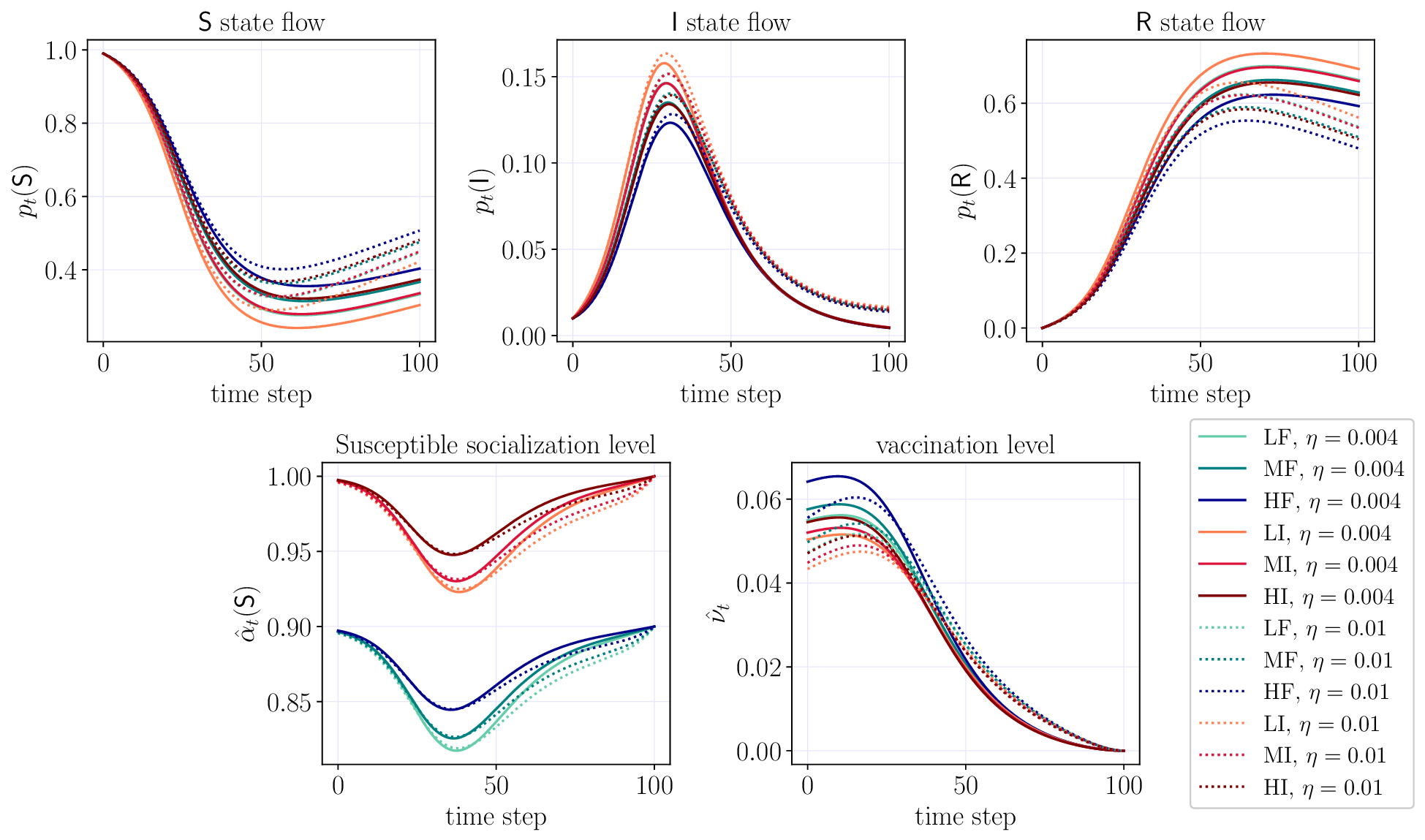}
        \label{fig:supp_eta}
    }
    \captionsetup{font=footnotesize}
    \caption{
    Sensitivity analysis on SIR model parameters. Left: comparison of epidemics mitigation effects under two treatment costs: \textbf{group adaptive}: $c_{\sI}^{\text{LF}} = c_{\sI}^{\text{LI}} = 1.05$, $c_{\sI}^{\text{MF}} = c_{\sI}^{\text{MI}} = 1$, $c_{\sI}^{\text{HF}} = c_{\sI}^{\text{HI}} = 0.8$; \textbf{low $c_\sI$}: $c_{\sI}^k = 0.8$ for all groups. Right: comparison of epidemics mitigation effects under two waning rates: $\eta^{k} = 0.004$; $\eta^{k} = 0.01$ for all groups. 
    }
    \label{fig:ci_eta_comparison}
\end{figure*}

\section{Extension to SIRD-type Model to Include Deceased Status}
\subsection{Model and Equilibrium Characterization}
We also introduce an extension of our model that includes the \textit{Deceased} ($\sD$) state for a possible health status, in other words, we extend our SIR-type model to SIRD-type model. Similar to SIR-type model, the state transitions of the representative individual in group $k\in[K]$ will be modeled by using a continuous-time Markov chain with the following transition rate matrix to include the group-specific \textit{mortality rate}, $\rho^k \in[0,1]$:
\begin{equation}
\label{eq:q_matrix}
\begin{blockarray}{cccccc}
& \sS & \sI & \sR & \sD \\[0.5mm]
\begin{block}{c[ccccc]}
  \sS & \cdots  & \beta^k \alpha_{t}^{k} Z_{t}^{k} & \kappa^k \nu_t^k & 0\\
  \sI & 0  & \cdots & (1-\rho^k)\gamma^k & \rho^k \gamma^k \\
  \sR & \eta^k & 0 & \cdots & 0 \\
  \sD & 0 & 0 & 0 & \cdots \\
\end{block}
\end{blockarray},
\end{equation}
In the extended model an individual (in group $k \in[K]$ in state $\sI$ transitions to either state $\sR$ or $\sD$ after an exponentially distributed time with rates $(1-\rho^k)\gamma^k$ and $\rho^k\gamma^k$, whichever randomly sampled time comes first. We notice that if $\rho^k=0$, the extended model recovers the SIR-type model introduced in the paper.

The individual objectives of the individuals in states $\sS, \sI$, and $\sR$ stays the same. We add a terminal cost for state $\sD$ that is mathematically modeled as with the addition of the term $\mathcal{D}^k \mathbf{1}_{\sD}(X_T^k)$ in the cost functionals. This means that the cost functional for an individual who is authority follower $k\in\{\lf, \mf, \hf\}$ will become:
\begin{equation*}
\begin{aligned}
    &J^k(\balpha^k, \bnu^k; \bZ) = \\ &\qquad \mathbb{E}\Big[\int_0^T \big[\big(c^k_{\lambda} (1 -\alpha_t^k)^2 + c_{\nu}^k (\nu_t^k)^2 \big)\mathbf{1}_{\sS}(X^k_t) 
    + \big((\xi_{\sI} -\alpha_t^k)^2 + c_{\sI}^k\big)\mathbf{1}_{\sI}(X^k_t)+ (1 -\alpha_t^k)^2\mathbf{1}_{\sR}(X^k_t)\big] dt \underbrace{+ \mathcal{D}^k \mathbf{1}_{\sD}(X_T^k)}_{\text{Terminal cost}}\Big]. 
\end{aligned}
\end{equation*}
For the authority indifferent individual the cost functional is defined similarly by adding the terminal cost in their original cost functional. We stress that we do not define social distancing or vaccination controls for the agents who are in state $\sD$ due to its practical meaning.

\begin{theorem}
\label{the:fbode_sird}
    Under assumption~\ref{assu:fbode}, multi-population mean field Nash equilibrium controls are given as
    \begin{align}
        \hat{\alpha}_t^k(\sS) &= \lambda_t^{k, \sS} + \frac{\beta^k Z_t^k \big(u_t^k(\sS) - u_t^k(\sI)\big)}{2c_\lambda^k} \qquad\qquad&& \hat{\alpha}_t^k(\sS) = 1 + \frac{\beta^k Z_t^k \big(u_t^k(\sS) - u_t^k(\sI)\big)}{2c_\lambda^k} \label{eq:alphaS_d}\\
        \hat{\alpha}_t^{k}(\sI) &= \lambda_t^{k, \sI} && \hat{\alpha}_t^{k}(\sI) =\xi_I\label{eq:alphaI_d}\\
        \hat{\alpha}_t^{k}(\sR) &= \lambda_t^{k, \sR}\label{eq:alphaR_d} &&\hat{\alpha}_t^{k}(\sR)=1\\
        \hat{\nu}^k_t &= \underbrace{\frac{\kappa^k \big(u_t^k(\sS) - u_t^k(\sR)\big)}{2c_\nu^k}}_{\textcolor{black}{k\:\in\:\{{\lf,\: \mf,\: \hf}}\}} && \hat{\nu}^k_t= \underbrace{\frac{\kappa^k \big(u_t^k(\sS) - u_t^k(\sR)\big)}{2c_\nu^k}}_{\textcolor{black}{k\:\in\:\{{\li,\: \mi,\: \hi}}\}}\label{eq:Nu_d}
    \end{align}
    for all $k \in[K]$ if $(\bp,\bu)=(\bp^k, \bu^k)_{k\in[K]}$ tuple solves the following FBODE system:
\begin{align*}
\label{eq:FBODE_sird}
    \dot p^k_t(\sS) &= -\beta^k \hat{\alpha}^{k}_t(\sS) Z_t^k p_t^k(\sS) - \kappa^k \hat{\nu}_t^kp_t^k(\sS) + \eta^k p_t^k(\sR) ,\\[1mm]
    \dot p^k_t(\sI) &= \beta^k \hat{\alpha}^{k}_t(\sS) Z_t^k p_t^k(\sS) - \gamma^k p_t^k(\sI),\\[1mm]
    \dot p^k_t(\sR) &= \textcolor{black}{(1-\rho^k)}\gamma^k p_t^k(\sI) +\kappa^k \hat{\nu}_t^kp_t^k(\sS) - \eta^k p_t^k(\sR),
    \\[1mm]
    \textcolor{black}{\dot p^k_t(\sD)} &= \textcolor{black}{\rho^k\gamma^k p_t^k(\sI),}
    \\[1mm]
    \dot u^{k_1}_t (\sS)&= \beta^{k_1} \hat\alpha^{k_1}_t(\sS)Z^{k_1}_t\big(u^{k_1}_t(\sS) - u^{k_1}_t( \sI)\big) + \kappa^{k_1} \hat{\nu}_t^{k_1} \big(u^{k_1}_t(\sS) - u^{k_1}_t( \sR)\big) 
    -c^{k_1}_\lambda \big(\lambda^{{k_1}, \sS}_t-\hat\alpha^{k_1}_t(\sS)\big)^2 - c_\nu^{k_1} (\hat{\nu}_t^{k_1})^2,
    \\[1mm]
    \dot u^{k_2}_t (\sS)&= \beta^{k_2} \hat\alpha^{k_2}_t(\sS)Z^{k_2}_t\big(u^{k_2}_t(\sS) - u^{k_2}_t( \sI)\big) + \kappa^{k_2} \hat{\nu}_t^{k_2} \big(u^{k_2}_t(\sS) - u^{k_2}_t( \sR)\big) 
    -c^{k_2}_\lambda \big(1-\hat\alpha^{k_2}_t(\sS)\big)^2 - c_\nu^{k_2} (\hat{\nu}_t^{k_2})^2,
    \\[1mm]
    \dot u^k_t (\sI)&= \textcolor{black}{(1-\rho^k)}\gamma^k \big(u^k_t(\sI) - u^k_t( \sR)\big) \textcolor{black}{+\rho^k\gamma^k \big(u^k_t(\sI) - u^k_t( \sD)\big)} - c^k_{\sI},
    \\[1mm]
    \dot u^k_t (\sR)&= \eta^k (u_t^k(\sR)-u_t^k(\sS)),
    \\[1mm]
    \textcolor{black}{\dot u^k_t (\sD)}&= \textcolor{black}{0,}
    \\[1mm]
     u^k_T(e)&= 0,\quad \forall e\in \{\sS, \sI, \sR\}, \quad u^k_T(e) = \mathcal{D}^k, \qquad p^k_0(e) = p^k(e),\quad \forall e\in \{\sS, \sI, \sR, \textcolor{black}{\sD}\},  \\[1mm]
     Z^k_t &= \sum_{l \in [K]} w(k, l) \EE[\hat{\alpha}_t^l(\sI)] p_t^l(\sI) m^l, \quad \forall k \in [K],\: \forall k_1\in\{\lf,\mf, \hf\},\: \forall k_2\in\{\li, \mi, \hi\},\: \forall t\in[0,T].
\end{align*}
\end{theorem}

The proof of the theoretical characterization of the multi-population mean field Nash equilibrium in the extended SIRD model with an FBODE system that is presented in~\ref{the:fbode_sird} follows analogously to the SIR case (see proof of Theorem~\ref{the:fbode_sir}), and thus we omit the full derivation for brevity. The existence and uniqueness results, as well as the numerical algorithm used to compute the equilibrium, also extend naturally from the SIR setting with minor modifications. For consistency and comparison, we include the same set of numerical experiments as in the SIRD model, focusing particularly on those that align closely with the SIR experiments presented in the main text.

\subsection{Experiments on SIRD-type model}
For the SIRD model, two additional death-related parameters are introduced. The baseline mortality rate is set to $\rho^{k} = 0.005 (0.5\%)$, and the baseline mortality cost is $\mathcal{D}^{k} = 80$ for all $k\in[K]$. We extend Algorithm~\ref{algo:vacc} by incorporating the dynamics associated with the $\sD$ state and replicate the previous experiments under the SIRD setting to examine the impact of including this additional state.

\begin{figure*}[h]
    \centering
    \subfloat[\textbf{Permissive Policy vs. Adaptive Policy}]{
        \includegraphics[width=0.45\textwidth]{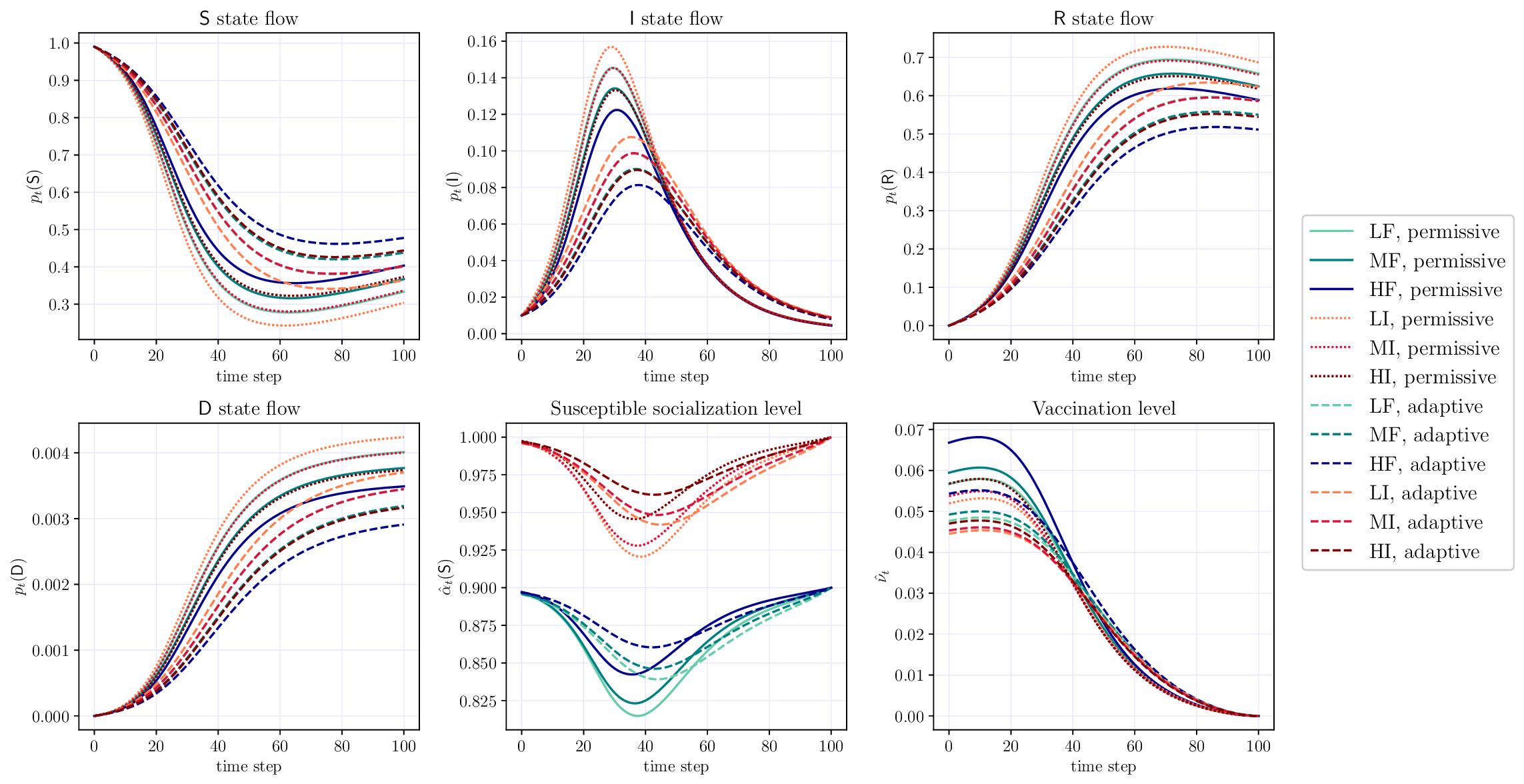}
    }
    \hfill
    \subfloat[\textbf{Permissive Policy vs. Strict Policy}]{
        \includegraphics[width=0.45\textwidth]{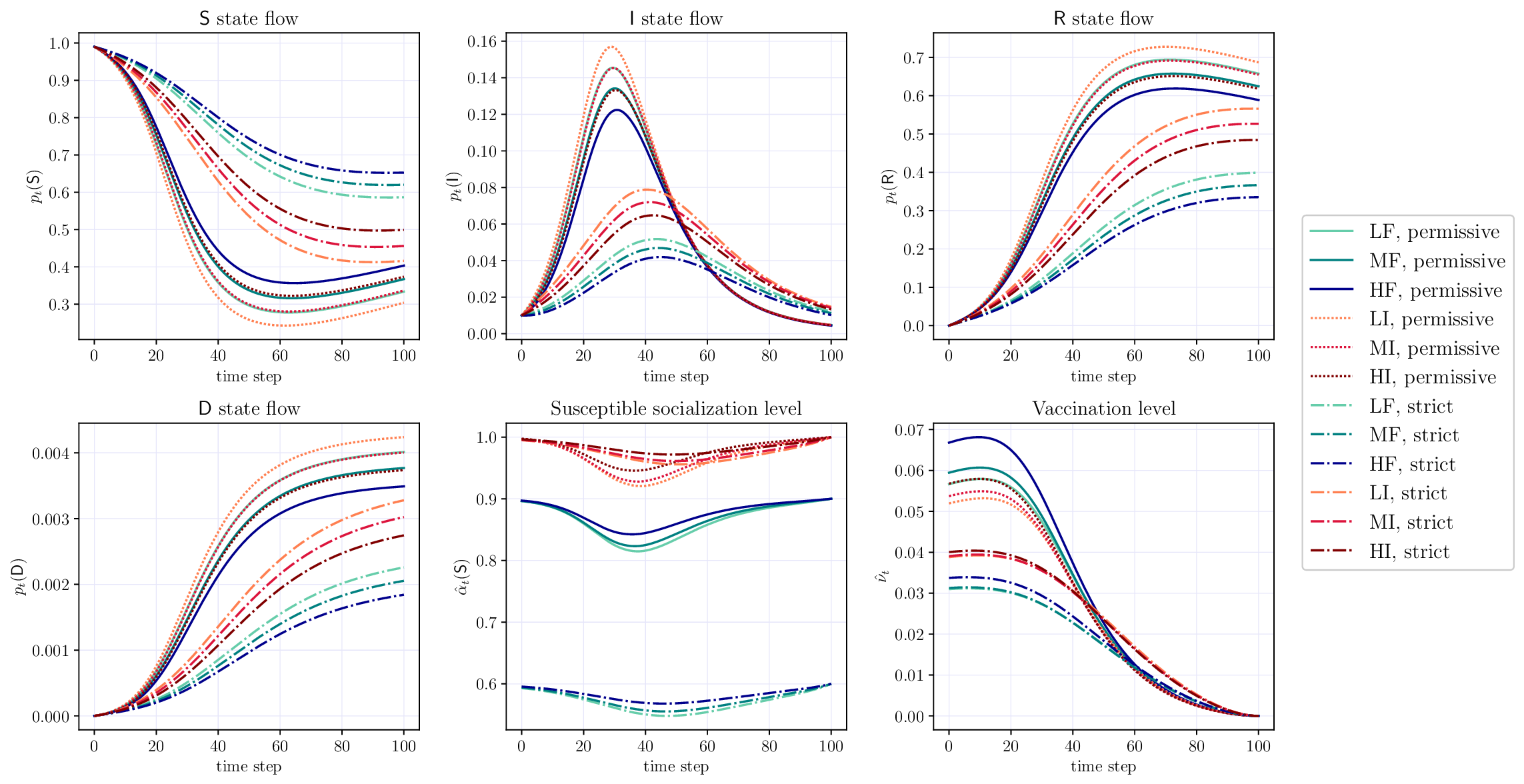}
    }
    \captionsetup{font=footnotesize}
    \caption{
        Pairwise comparison of epidemics mitigation effects under three policies for SIRD model. \textbf{Permissive}: $\lambda_t^{k, \sS} = 0.9, \lambda_t^{k, \sI} = 0.9, \lambda_t^{k, \sR} = 0.9$ for all groups; \textbf{Strict}: $\lambda_t^{k, \sS} = 0.6, \lambda_t^{k, \sI} = 0.6, \lambda_t^{k, \sR} = 0.6$ for all groups; \textbf{Adaptive}: $\lambda_t^{k, \sS} = 0.9, \lambda_t^{k, \sI} = 0.6, \lambda_t^{k, \sR} = 0.9$ for all groups.
    }
    \label{fig:policy_comparison_sird}
\end{figure*}

\begin{figure*}[h]
    \centering
    \subfloat[\textbf{Sensitivity analysis on $\rho$}]{
        \includegraphics[width=0.45\textwidth]{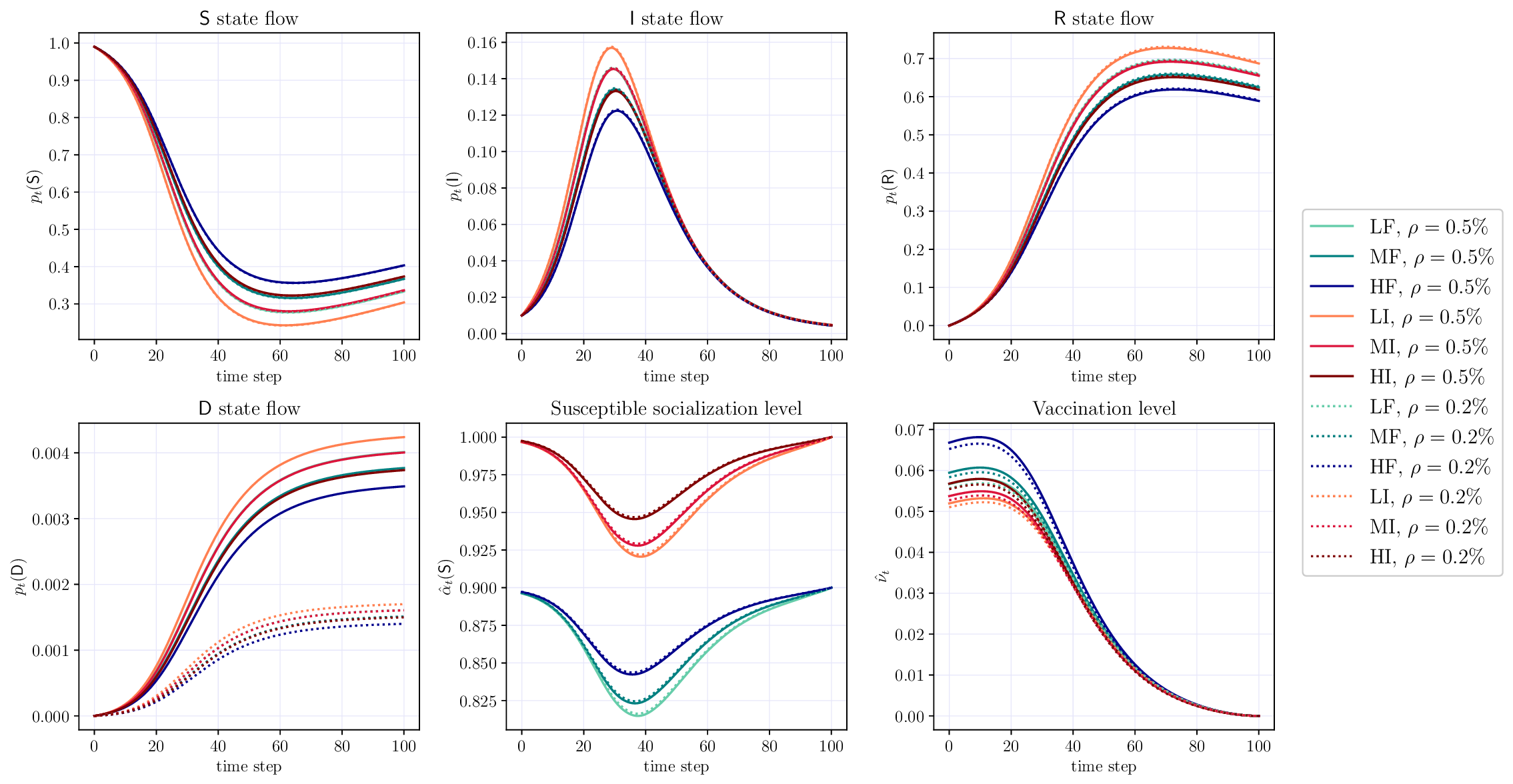}
        \label{fig:supp_rho_death}
    }
    \hfill
    \subfloat[\textbf{Sensitivity analysis on $\mathcal{D}$}]{
        \includegraphics[width=0.45\textwidth]{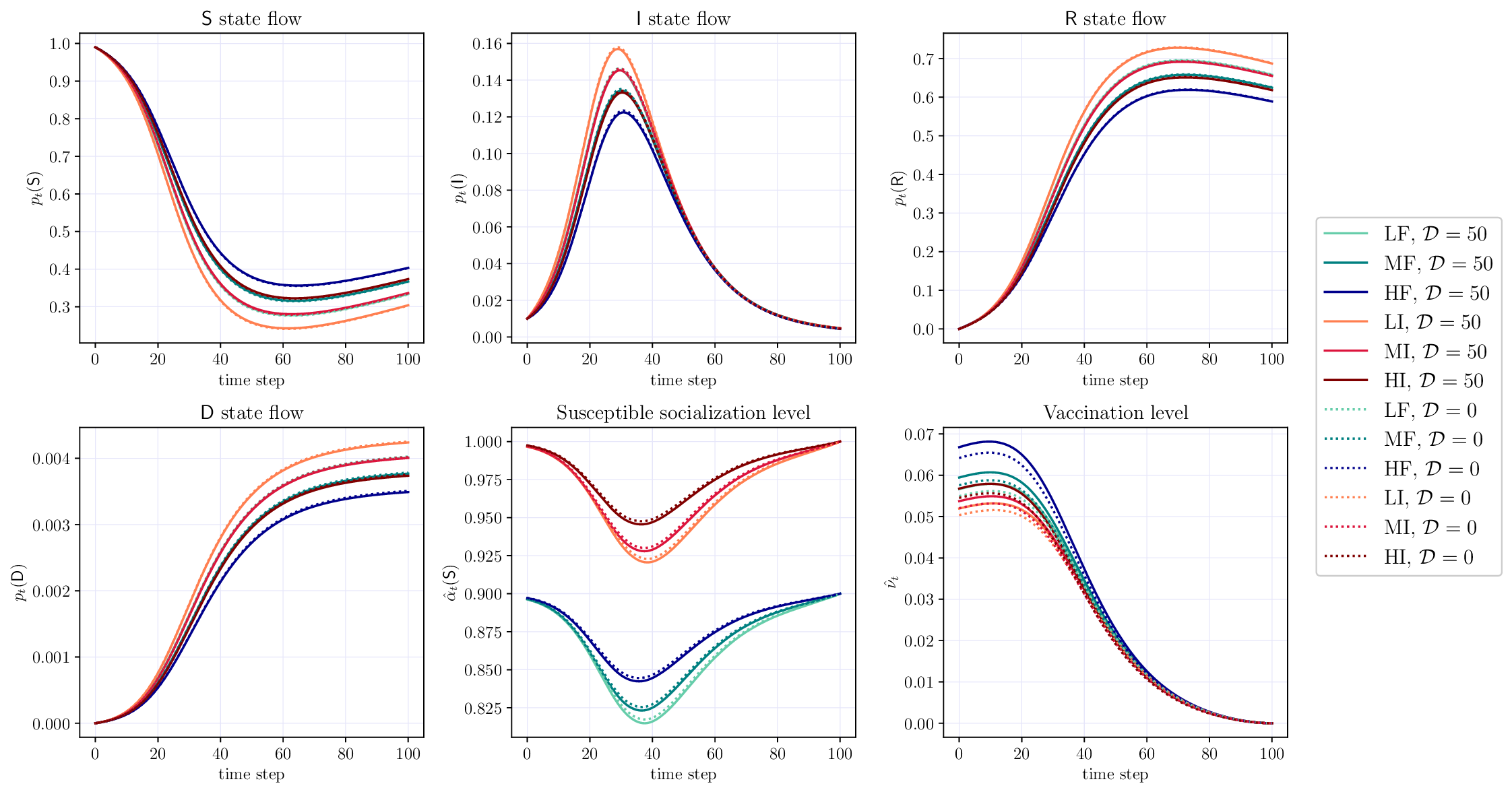}
        \label{fig:supp_D_death}
    }
    \captionsetup{font=footnotesize}
    \caption{
    Analysis on death state related parameters for SIRD model. Left: comparison of epidemics mitigation effects under two morality rates: $\rho^k = 0.5 \%$; $\rho^k = 0.2 \%$ for all groups. Right: Comparison of epidemics mitigation effects under two morality costs: $\cD^k = 80$; $\cD^k = 0$ for all groups.
    }
    \label{fig:rho_D_comparison_sird}
\end{figure*}

Figure~\ref{fig:policy_comparison_sird} pairwise compares the effects of the strict and adaptive social-distancing policies with the baseline permissive policy. The adaptive policy effectively suppresses the infection outbreak by reducing peak infection levels relative to the permissive one; however, even if the resulting proportion of deceased individuals is decreased, it stays at higher levels. In contrast, the strict policy leads to a notable reduction in overall deceased rates, with an even lower mortality observed among authority followers. This suggests that consistently enforced social interaction restrictions, though in potential more economically costly, are considerably more effective in mitigating mortality than adaptive or permissive policies.

Figure~\ref{fig:supp_rho_death} compares outcomes under two different mortality rates (i.e., sensitivity analysis on $\rho$). The proportion of deceased individuals increases approximately in proportion to the mortality rate parameter, with lower-income and authority indifferent groups having larger death proportions. Under a higher mortality rate, getting infected becomes more risky, as each infection carries a higher probability of death. Consequently, individuals are more inclined to vaccinate to avoid infection, while increased levels of vaccination encourages slightly higher socialization levels. The infection level decreases slightly, suggesting that comparisons between two lower mortality rates influence individual vaccination behavior more than the epidemic transmission process.

Figure~\ref{fig:supp_D_death} compares outcomes under two mortality cost settings (i.e., sensitivity analysis on $\mathcal{D}$). For a similar reason as above, when the mortality cost is higher, individuals exhibit greater vaccination uptake while slightly increasing their socialization levels. However, the resulting death proportions remain close to each other, indicating that more perceived panic of death does not substantially affect the overall mortality outcome, which is primarily determined by the intrinsic mortality rate of the infection.

\end{document}